\documentclass[lectures]{pcms-l-mod}

\usepackage{amssymb}
\usepackage{tikz}
\usetikzlibrary{patterns,arrows}
\usepackage{booktabs}
\usepackage{url}

\newcommand{\R}{{\mathbb R}}

\newcommand{\C}{{\mathbb C}}

\newcommand{\Z}{{\mathbb Z}}

\newcommand{\mC}{{\mathcal{C}}}

\newcommand{\vol}{\mathop{\textup{vol}}}

\newcommand{\hatbigoplus}{\mathop{\widehat{\bigoplus}}\limits}

\theoremstyle{plain}
\newtheorem{theorem}{Theorem}[section]
\newtheorem{lemma}[theorem]{Lemma}
\newtheorem{proposition}[theorem]{Proposition}
\newtheorem{corollary}[theorem]{Corollary}
\newtheorem{conjecture}[theorem]{Conjecture}

\theoremstyle{definition}

\newtheorem{exercise}[theorem]{Exercise}

\theoremstyle{remark}

\numberwithin{equation}{section}

\begin{document}

\frontmatter

\tableofcontents

\mainmatter

\LectureSeries{Packing, coding, and ground states \author{Henry Cohn}}
\address{Microsoft Research New England, One Memorial Drive, Cambridge, MA 02142}
\email{cohn@microsoft.com}

\section*{Preface}

In these lectures, we'll study simple models of materials from several
different perspectives: geometry (packing problems), information theory
(error-correcting codes), and physics (ground states of interacting particle
systems).  These perspectives each shed light on some of the same problems
and phenomena, while highlighting different techniques and connections.

One noteworthy phenomenon is the exceptional symmetry that is found in
certain special cases, and we'll examine when and why it occurs.  The
overall theme of the lectures is thus order vs.\ disorder.  How much
symmetry can we expect to see in optimal geometric structures?

The style of these lecture notes is deliberately brief and informal. See
Conway and Sloane's book \emph{Sphere packing, lattices and groups}
\cite{ConwaySloane1999} for far more information about many of the
mathematical objects we'll discuss, as well as the references cited in the
notes for omitted details.

I've included a dozen exercises for the reader, which cover things I think
it's most useful to do for oneself.  The exercises vary in difficulty, from
routine verifications to trickier computations.  There's no need to solve
them if you are willing to take a few things on faith, but I highly
recommend engaging actively with this material, and the exercises would be a
good way to get started.

These notes are based on my PCMI lectures from 2014 and were written before
Viazovska \cite{Viazovska2016} found a remarkable solution to the sphere
packing problem in $\R^8$ using linear programming bounds. The only updates
to reflect this development are a few footnotes.

\section*{Acknowledgments}

I am grateful to Matthew de Courcy-Ireland for serving as the teaching
assistant for this course and for providing feedback on the lecture notes.

\lecture{Sphere packing}

\section{Introduction}

The sphere packing problem asks for the densest packing of congruent spheres
in $\R^n$.  In other words, how can we cover the greatest fraction of space
using congruent balls that do not overlap (i.e., that have disjoint
interiors)? The \emph{density} is the fraction of space covered.  Finding
the densest sphere packing sounds simple, but it turns out to be a
surprisingly deep and subtle problem.

Before we dive into the sphere packing problem, it's worth thinking about
how to write down a rigorous definition. Although pathological packings may
not have well-defined densities, everything we could reasonably hope for is
true: we can define the optimal density by taking a suitable limit, and
there is a packing that achieves this density. Specifically, given a packing
$\mathcal{P}$, a point $x \in \R^n$, and a positive real number $r$, let
\[
\Delta_{r,x}(\mathcal{P}) = \frac{\vol (B_r(x) \cap \mathcal{P})}{\vol B_r(x)}
\]
be the fraction of the ball $B_r(x)$ of radius $r$ centered at $x$ that is
covered by $\mathcal{P}$.  If we define the \emph{optimal packing density}
$\Delta_n$ in $\R^n$ by
\[
\Delta_n = \limsup_{r \to \infty} \sup_{\mathcal{P}} \Delta_{r,0}(\mathcal{P}),
\]
then there exists a single packing $\mathcal{P}$ for which
\[
\lim_{r \to \infty} \Delta_{r,x}(\mathcal{P}) = \Delta_n
\]
uniformly for all $x \in \R^n$.  See \cite{Groemer1963} for a proof.

What are the optimal sphere packings in low dimensions?  In one dimension,
we have the interval packing problem on the line, which is trivial.  In two
dimensions, the answer is pretty clearly the hexagonal packing, with each
disk surrounded by six others:
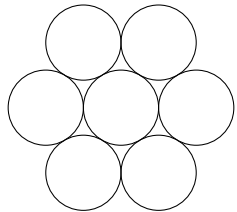
\begin{center}
\begin{tikzpicture}
  \draw (0,0) circle (0.5);
  \draw (1,0) circle (0.5);
  \draw (-1,0) circle (0.5);
  \draw (0.5,0.866025) circle (0.5);
  \draw (0.5,-0.866025) circle (0.5);
  \draw (-0.5,0.866025) circle (0.5);
  \draw (-0.5,-0.866025) circle (0.5);
\end{tikzpicture}
\end{center}
However, proving optimality takes a genuine idea.  For example, one can show
that the Voronoi cells (the sets of points closer to each sphere center than
to the others) in the hexagonal packing are as small as possible in any
packing. See \cite{Thue1892} for the first proof of optimality,
\cite{Rogers1958,FejesToth1972} for subsequent proofs, and \cite{Hales2000}
for a particularly short proof.

In three dimensions, the sphere packing problem is much more difficult.
There is a natural guess for the solution, namely stacking hexagonal layers
as densely as possible, so that each is nestled into the gaps in the
neighboring layers.  Such a packing is known to be optimal, via an elaborate
proof \cite{Hales2005} that depends on computer calculations.  The original
proof was so long and complex that it was difficult to check carefully, but
it has recently been verified at the level of formal logic \cite{Flyspeck}.

In four or more dimensions, the optimal sphere packing density is not known,
although there are upper and lower bounds.\footnote{Viazovska
\cite{Viazovska2016} has recently solved the sphere packing problem in $\R^8$
using linear programming bounds.}

\begin{exercise}
How can hexagonal layers be stacked to form dense packings in $\R^3$?  Show
that there are an uncountable number of different ways to do so, even if you
consider two packings the same when they are related by a rigid motion of
space. Can you extend this analysis to $\R^4$?  Which packings can you get
by stacking optimal three-dimensional packings as densely as possible? How
many can you find? How dense are they? What about $\R^5$? $\R^6$?  How high
can you go?

Feel free to give up after four dimensions, but the further you go, the more
interesting phenomena you'll run into.  By $\R^{10}$, this iterated stacking
process will no longer produce the densest possible sphere packings, but
nobody knows whether it fails before that.  See \cite{ConwaySloane1995} for
more details on what happens in dimensions two through ten.
\end{exercise}

\section{Motivation}

There are several reasons why we should care about sphere packing.  One is
that it's a natural geometric problem: humanity ought to know the answer to
such a simple and natural question.

Another reason is that the problem has interesting solutions.  Sometimes
it's difficult to judge how interesting a problem is in the abstract, before
taking a look at the phenomena that occur.  Sphere packing is full of rich,
intricate structures that are themselves of intrinsic interest, and this
makes the problem far more appealing than it would have been if the answers
had been less exciting.

A third reason to care about sphere packing is that it is a toy model of
granular materials.  Of course no real material consists of identical perfect
spheres, and the sphere packing problem also neglects forces and dynamics.
However, sphere packing is at least a first step towards understanding the
density of an idealized material.  (See \cite{Loewen2000} for a statistical
physics perspective on packing.)

The most important practical reason to study sphere packing is also one of
the most surprising reasons: high-dimensional sphere packings are essential
for communication over noisy channels, as we'll spend the rest of this
section understanding.  This is really an assertion about information
theory, Claude Shannon's great discovery from his famous 1948 paper \emph{A
mathematical theory of communication} \cite{Shannon1948}.  Sphere packing
per se again deals with an idealized scenario, but it illustrates some of
the fundamental principles underlying information theory.

The setting works as follows.  Suppose we are sending messages over some
communication channel.  We will represent the signals by points in a bounded
subset of $\R^n$, say the ball of radius $R$ about the origin (the precise
subset is not so important).  In this model, each coordinate represents some
measurement used to describe the signal.  For example, for a radio signal we
could measure the amplitude at different frequencies.  There is no reason to
expect the number of measurements to be small, and realistic channels can
involve hundreds or even thousands of coordinates.  Thus, the signal space
for the channel will be high-dimensional.  Note that this dimensionality has
nothing to do with the physical space we are working in; instead, it simply
represents the number of independent measurements we make on the signals.

Each signal will be an individual transmission over the channel at a given
time, and we will send a stream of signals as time passes.  Of course, the
big difficulty with communication is noise: if we send a signal $s$, then
the received signal $r$ at the other end will generally not be exactly equal
to $s$.  Instead, it will have been perturbed by channel noise.  In a useful
channel, the noise level will be fairly low, and we can expect that $|r-s| <
\varepsilon$ for some fixed $\varepsilon$ (the noise level of the channel).

Thus, we can imagine an open \emph{error ball} of radius $\varepsilon$ about
each signal sent, which shows how it could be received after adding noise:
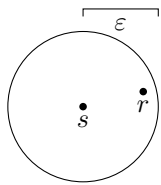
\begin{center}
\begin{tikzpicture}
  \draw (1,2.2) -- (1,2.3) -- (2,2.3) -- (2,2.2);
  \draw (1.5,2.3) node[below] {$\varepsilon$};
  \draw (1,1) circle (1);
  \fill[black] (1,1) circle (0.05);
  \fill[black] (1.8,1.2) circle (0.05);
  \draw (1,1) node[below] {$s$};
  \draw (1.8,1.2) node[below] {$r$};
\end{tikzpicture}
\end{center}
This is a simplistic model of noise, since we assume that the noise has no
directionality or structure, and that every perturbation up to radius
$\varepsilon$ could plausibly occur but nothing beyond that limit.  In
practice, engineers use more sophisticated noise models; for example, cell
phones have to take into account all sorts of other phenomena, such as
interference from reflected signals.  However, our basic noise model is a
good illustration of the essential principles.

How can we arrange our communications so as to remove the effects of noise?
We will build a vocabulary $S \subseteq \R^n$ of signals and only send
signals in $S$.  This is called an \emph{error-correcting code}. If two
distinct signals $s_1,s_2 \in S$ satisfy $|s_1-s_2| < 2\varepsilon$, then
the received signal could be ambiguous:
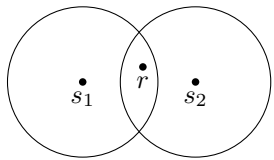
\begin{center}
\begin{tikzpicture}
  \draw (1,1) circle (1);
  \draw (2.5, 1) circle (1);
  \fill[black] (1,1) circle (0.05);
  \fill[black] (2.5,1) circle (0.05);
  \fill[black] (1.8,1.2) circle (0.05);
  \draw (1,1) node[below] {$s_1$};
  \draw (2.5,1) node[below] {$s_2$};
  \draw (1.8,1.2) node[below] {$r$};
\end{tikzpicture}
\end{center}
Therefore, we will keep all signals in $S$ at least $2\varepsilon$ apart, so
that the error balls are disjoint:
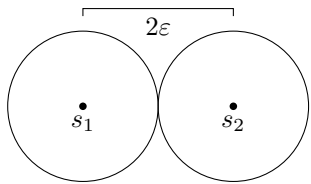
\begin{center}
\begin{tikzpicture}
  \draw (1,1) circle (1);
  \draw (3, 1) circle (1);
  \fill[black] (1,1) circle (0.05);
  \fill[black] (3,1) circle (0.05);
  \draw (1,1) node[below] {$s_1$};
  \draw (3,1) node[below] {$s_2$};
  \draw (1,2.2) -- (1,2.3) -- (3,2.3) -- (3,2.2);
  \draw (2,2.3) node[below] {$2\varepsilon$};
\end{tikzpicture}
\end{center}

This is exactly the sphere packing problem.  We want the signal set $S$ to
be as large as possible, since having more signals available increases the
rate at which we can transmit information, but the $\varepsilon$-balls about
the signals in $S$ are not allowed to overlap.  How large can we make $S$
subject to this constraint?  Recall that the only available subset of $\R^n$
in our model is the ball of radius $R$.  Thus, the question becomes how many
$\varepsilon$-balls we can pack into a ball of radius $R+\varepsilon$.  (The
radius is $R+\varepsilon$, rather than $R$, because the error balls can
stick out over the edge.)  That's a finite version of sphere packing, and we
recover the usual version in all of $\R^n$ in the limit when $R$ is much
larger than $\varepsilon$.  That limit is exactly the situation we expect,
since the channel is not very useful if $\varepsilon$ is on the same scale
as $R$.

It is remarkable that although high-dimensional packing sounds utterly
abstract and impractical, it turns out to be particularly important for
applications.  In these lectures we will focus on the theory behind sphere
packing, rather than the applications, but it is helpful to keep in mind
that a high-dimensional packing is a tool for communicating over a noisy
channel.

\section{Phenomena}

Relatively little is understood about the sphere packing problem.  One might
hope for a systematic solution that works in every dimension, but that just
doesn't seem possible.  Instead, each dimension has its own idiosyncrasies.
Getting a feeling for how $\R^8$ differs from $\R^7$ or $\R^9$ is part of
the charm of the subject, but these differences mean the packing problem is
much more subtle than it sounds.  In two or three dimensions, we can rely on
our spatial intuition and summarize the procedure as ``just do the obvious
thing,'' but there is no obvious thing to do in $\R^n$.

Good constructions are known in low dimensions, and there is little doubt
that humanity has found the optimal density through at least the first eight
dimensions.  However, we have absolutely no idea what the best
high-dimensional packings look like.  For example, we do not know whether to
expect them to be ordered and crystalline, or disordered and pseudorandom.
Many researchers expect disorder, perhaps on the grounds that this is the
default when there is no reason to expect order.  However, we lack the
theoretical tools to analyze this question.

All we know in general are upper and lower bounds for the optimal density,
and these bounds are distressingly far apart.  For example, in $\R^{36}$
they differ by a multiplicative factor of $58$: if you take the densest
known packing in $\R^{36}$, then the best we can say is that you couldn't
fit in any more than $58$ times as many spheres if you rearranged them.  The
ratio of the upper and lower bounds in $\R^n$ grows exponentially as $n \to
\infty$.

At first, this gap sounds absurd.  How could our bounds possibly be off by
an exponential factor?  One way to think about it is that volume scales
exponentially in high dimensions, because the volume of a hypercube of side
length $\ell$ in $\R^n$ is $\ell^n$, which is exponential in $n$. If you
take a packing in $\R^n$ and move the sphere centers 1\% further apart, then
you lower the density by a factor of $1.01^n$.  In low dimensions this
factor is insignificant, but in high dimensions it is enormous.  Thus, even
a little bit of uncertainty in the sphere locations translates to an
exponential uncertainty in the density.

On a scale from one to infinity, a million is small, but we know almost
nothing about sphere packing in a million dimensions.  The best we can say
is that the optimal density is at least a little larger than $2^{-1000000}$.
More generally, the following greedy argument gives a surprisingly easy
lower bound of $2^{-n}$ in $\R^n$.

Consider a \emph{saturated} packing in $\R^n$, i.e., a packing such that no
further spheres can be added without overlap.  Such packings certainly
exist, because one can obtain a saturated packing by iteratively adding
spheres as close to the origin as possible. Alternatively, there are
saturated packings on flat tori because there is room for only finitely many
spheres, and unrolling such a packing yields a saturated periodic packing in
Euclidean space.

\begin{proposition} \label{prop:greedy}
Every saturated sphere packing in $\R^n$ has density at least $2^{-n}$.
\end{proposition}

\begin{proof}
No point in $\R^n$ can have distance at least $2$ from all the sphere
centers in a saturated packing with unit spheres, because we could center a
new sphere at such a point without creating any overlap (see
Figure~\ref{figure:greedy}). In other words, doubling the radius of the
spheres in a saturated packing would cover space completely.  Doubling the
radius increases the volume by a factor of $2^n$, and so the original
spheres must occupy at least a $2^{-n}$ fraction of $\R^n$.  Thus, every
saturated packing has density at least $2^{-n}$.
\end{proof}

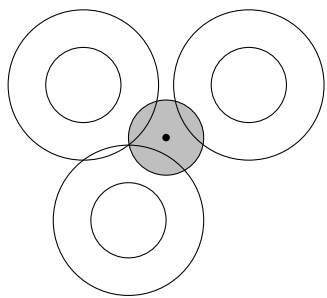
\begin{figure}
\begin{center}
\begin{tikzpicture}
\fill[black!25] (3.1,1.3) circle (0.5); \draw (3.1,1.3) circle (0.5);
\draw (2,2) circle (0.5); \draw (2,2) circle (1);
\draw (4.2,2) circle (0.5); \draw (4.2,2) circle (1);
\draw (2.6,0.2) circle (0.5); \draw (2.6,0.2) circle (1);
\fill[black] (3.1,1.3) circle (0.05);
\end{tikzpicture}
\end{center}
\caption{Any point not covered by the double-radius spheres could be used as
the center of a new sphere (shaded above).} \label{figure:greedy}
\end{figure}

In $\R^1$ there are saturated packings with density arbitrarily close to
$1/2$, but that is the only case in which Proposition~\ref{prop:greedy} is
sharp, because the bound is sharp exactly when the double-radius balls can
tile $\R^n$. One way to improve it is to prove a lower bound for how
inefficient a sphere covering in $\R^n$ must be.  For example, using the
Coxeter-Few-Rogers theorem on sphere covering \cite{CoxeterFewRogers1959}
improves the bound to $e^{-3/2} n \cdot 2^{-n}$ asymptotically.

At first $2^{-n}$ sounds like a rather weak bound, which must be far from
the truth.  However, nobody has been able to obtain any exponential
improvement to it, and perhaps it is closer to the truth than one would
guess.  In any case, it is nearly all we know regarding density lower bounds
in high dimensions.  A long sequence of improvements ground to a halt with
Ball's bound of $2(n-1) \cdot 2^{-n}$ in 1992 \cite{Ball1992}, before
progress began again nearly twenty years later.  Vance proved a lower bound
asymptotic to $6n/e \cdot 2^{-n}$ in 2011 \cite{Vance2011}, which improves
on Ball's bound because $e<3$, and Venkatesh followed that with a much
larger constant-factor improvement as well as a bound proportional to $n
\log \log n \cdot 2^{-n}$ for a certain sparse sequence of dimensions
\cite{Venkatesh2013}.  This last bound is particularly exciting because it
is the first superlinear improvement on $2^{-n}$, but on an exponential
scale all of these improvements are small.  For comparison, the best upper
bound known is $2^{-(0.5990\ldots+o(1))n}$, due to Kabatiansky and
Levenshtein \cite{KabatianskyLevenshtein1978} in 1978, with a
constant-factor improvement by Cohn and Zhao \cite{CohnZhao2014} in 2014.

Note that the greedy argument is nonconstructive, and the same is true of the
improvements mentioned above.  For large $n$, all known bounds anywhere near
$2^{-n}$ are nonconstructive, while every packing anyone has described
explicitly is terrible in high dimensions.  For example, one natural attempt
is to center spheres of radius $1/2$ at the integer lattice points $\Z^n$.
That yields a packing of density
\[
\frac{\pi^{n/2}}{(n/2)! 2^n},
\]
and the factorial in the denominator ruins the density. (Note that when $n$
is odd, $(n/2)!$ means $\Gamma(n/2+1)$.)

There are some patterns in low dimensions, but they quickly stop working.
For example, natural generalizations of the face-centered cubic packing from
$\R^3$ work well in $\R^4$ and $\R^5$, but not in higher dimensions, as we
will see in the next section. In $\R^{10}$, the best packing known is based
on a periodic arrangement with $40$ spheres in each fundamental cell
\cite[p.~140]{ConwaySloane1999}.

Crystalline packings work beautifully in low dimensions, but they become
increasingly difficult to find in high dimensions.  Perhaps they just aren't
optimal?  It's natural to speculate about amorphous packings, but nobody
really knows.  In high dimensions, we can analyze only random or typical
packings, and we simply do not know how close they are to the very best.

One philosophical quandary is that too much structure seems to make
high-dimensional packings bad, but the known lower bounds all rely on some
sort of heavy structure. Vance's and Venkatesh's techniques give the best
density, but they involve the most structure, namely lattices with
nontrivial symmetry groups acting on them.  The trade-off is that structure
seemingly hurts density but helps in analyzing packings.

The most remarkable packings are the $E_8$ root lattice in $\R^8$ and the
Leech lattice $\Lambda_{24}$ in $\R^{24}$. They are incredibly symmetrical
and dense packings of spheres, and they must be optimal, although this has
not yet been proved.\footnote{Until very recently in \cite{Viazovska2016} for
$n=8$.} What makes them exciting is that they turn out to be connected with
many areas in mathematics and physics, such as string theory, hyperbolic
geometry, and finite simple groups.  See \cite{ConwaySloane1999} and
\cite{Ebeling2013} for more information about these wonderful objects, as
well as the next section for a construction of $E_8$.

\section{Constructions}

How can we form a sphere packing?  The simplest structure we could use is a
\emph{lattice}, the integer span of $n$ linearly independent vectors in
$\R^n$.  In other words, given a basis $v_1,\dots,v_n$, we center the
spheres at the points
\[
\{a_1v_1 + a_2v_2 + \dots +a_nv_n \ |\  a_1,\dots,a_n \in
\Z\}.
\]
The \emph{packing radius} of a lattice is half the shortest nonzero vector
length, since that is the largest radius for which the spheres do not
overlap.  Given a lattice basis $v_1,\dots,v_n$, the corresponding
\emph{fundamental cell} is the parallelotope
\[
\{x_1v_1 + x_2v_2 + \dots +x_nv_n \ |\  x_1,\dots,x_n \in
[0,1)\}.
\]
The translates of the fundamental cell by lattice vectors tile space.

In a lattice packing, there is one sphere per translate of the fundamental
cell, and the density is the volume ratio of the sphere and cell. More
generally, we could form a \emph{periodic packing}, which is the union of
finitely many translates of a lattice. Equivalently, there can be several
spheres per cell, which are then translated throughout space by the lattice
vectors. There is no reason to believe that one sphere per cell is the best
choice, and indeed periodic packings offer considerably more flexibility.

One confusing issue is that physicists use the term ``lattice'' to mean
periodic packing, while they call lattices ``Bravais lattices.''  We will
stick with the standard mathematical terminology.

There is no reason to believe that periodic packings achieve the greatest
possible density.  This is an open question above three dimensions, and it
is plausibly false in high dimensions.  However, periodic packings always
come arbitrarily close to the optimal density.  To see why, consider an
optimal packing, and imagine intersecting it with a large box. If we try to
repeat the part in the box periodically through space, then the only place
overlap could occur is along the boundary of the box.  We can fix any
problems by removing the spheres next to the boundary.  Shaving the packing
in this way produces a periodic packing without overlap, at the cost of
slightly lowering the density.  The decrease in density becomes arbitrarily
small if we use a sufficiently large box, and thus periodic packings come
arbitrarily close to the optimal packing density.

By contrast, lattices probably do not approach the optimal density in high
dimensions.  The problem is that unlike periodic packings, lattices have
limited flexibility.  A lattice is completely determined by a basis, and
thus a lattice in $\R^n$ can be specified by $n^2$ parameters (in fact,
fewer if we take the quotient by rigid motions).  Quadratically many
parameters just don't give enough flexibility to fill all the gaps in an
exponential amount of space.  It's natural to guess that when $n$ is large
enough, no lattice packing in $\R^n$ is ever saturated, but this conjecture
remains out of reach.

The best sphere packings currently known are not always lattice packings
($\R^{10}$ is the first case in which lattices seem to be suboptimal), but
many good packings are.  The simplest lattice is $\Z^n$, but it is a lousy
packing when $n>1$, as discussed above.  Instead, the ``checkerboard''
packing
\[
D_n = \{(x_1,\dots,x_n) \in \Z^n \ |\  x_1+\dots+x_n \textup{ is even}\}.
\]
is better for $n \ge 3$.  In fact, $D_3$, $D_4$, and $D_5$ are the best
packings known in their dimensions, and provably the best lattice packings
(see \cite{ConwaySloane1999} for more information). However, they are
suboptimal for $n \ge 6$.

What goes wrong for $n \ge 6$ is that the holes in $D_n$ grow larger and
larger. A \emph{hole} in a lattice $\Lambda$ in $\R^n$ is a point in $\R^n$
that is a local maximum for distance from the nearest point in $\Lambda$.
There are two classes of holes in $D_n$ for $n\ge3$, represented by
$(1,0,\dots,0)$, which is at distance $1$ from $D_n$, and
$(1/2,1/2,\dots,1/2)$, which is at distance
\[
\sqrt{\left(\frac12\right)^2 + \dots +\left(\frac12\right)^2} = \sqrt{\frac n4}.
\]
More generally, the translates of these points by $D_n$ are also holes, as
are the translates of $(1/2,1/2,\dots,1/2,-1/2)$.

When $n > 4$ we call $(1,0,\dots,0)$ a \emph{shallow hole} in $D_n$ and
$(1/2,\dots,\pm 1/2)$ a \emph{deep hole}, because $\sqrt{n/4} > 1$. When $n$
is large, the depth $\sqrt{n/4}$ of a deep hole is enormous.  For
comparison, note that the spheres in the $D_n$ packing have radius
$\sqrt{2}/2$, because the nearest lattice points are
\[(0,0,\dots,0) \qquad
\textup{and} \qquad (1,1,0,\dots,0),\] at distance $\sqrt{2}$. When $n$ is
large, the holes are much larger than the spheres in the packing, and $D_n$
is not even saturated, let alone an optimal packing.

This transition occurs at dimension eight, and something wonderful happens
right at the transition point.  When $n=8$, the radius $\sqrt{n/4}$ of a
deep hole equals the distance $\sqrt{2}$ between adjacent lattice points.
Thus, we can slip another copy of $D_8$ into the holes, which doubles the
packing density, and the new spheres fit perfectly into place.  The
resulting packing is called the \emph{$E_8$ root lattice}.

This construction of $E_8$ appears asymmetric, with two different types of
spheres, namely the original spheres and the ones that were added.  However,
they are indistinguishable, because $E_8$ is a lattice and thus all the
spheres are equivalent under translation.

\begin{exercise}
Check that $E_8$ is in fact a lattice.
\end{exercise}

The $E_6$ and $E_7$ lattices are certain cross sections of $E_8$.  The
$E_6$, $E_7$, and $E_8$ lattices are the densest lattice packings in $\R^6$
through $\R^8$, and they are almost certainly the densest sphere packings.

The Leech lattice $\Lambda_{24}$ in $\R^{24}$ is similar in spirit, but with
a more elaborate construction.  See \cite{Ebeling2013} for an elegant
treatment of the Leech lattice, as well as the theory of root lattices.

The \emph{kissing number} in $\R^n$ is the greatest number of spheres that
can touch a central sphere, if they all have the same size and cannot
overlap except tangentially.  It is known to be $6$ in $\R^2$, $12$ in
$\R^3$, $24$ in $\R^4$, $240$ in $\R^8$, and $196560$ in $\R^{24}$, but is
not known in any other dimensions.  The case of $\R^2$ is easy, but $\R^3$
is not \cite{SvdW1953}, and $\R^4$ is yet more difficult \cite{Musin2008}.
Surprisingly, $\R^8$ and $\R^{24}$ are quite a bit simpler than $\R^3$ or
$\R^4$ are \cite{OdlyzkoSloane1979,Levenshtein1979}, and we will settle them
in the fourth lecture.

\begin{exercise}
What are the shortest nonzero vectors in $D_n$?  In $E_8$? This will give
optimal kissing configurations in $\R^3$, $\R^4$, and $\R^8$.
\end{exercise}

\begin{exercise}
The vertices of a \emph{cross polytope} centered at the origin in $\R^n$
consist of $n$ pairs of orthogonal vectors of the same length (it's a
generalized octahedron).  Show how to decompose the vertices of a hypercube
in $\R^4$ into two cross polytopes. Find a symmetry of the hypercube that
interchanges them.
\end{exercise}

\begin{exercise}
Show how to decompose the minimal vectors in $D_4$ into three disjoint cross
polytopes, and find a symmetry of $D_4$ that cyclically permutes these cross
polytopes.
\end{exercise}

This symmetry is called \emph{triality}, and it makes $D_4$ more symmetrical
than any of its siblings.  When $n \ne 4$, the symmetries of $D_n$ are
simply permutations and sign changes of the coordinates, while $D_4$ has all
those plus triality.

\section{Difficulty of sphere packing}

Why is the sphere packing problem hard?  There are several reasons for this.
One is that there are many local optima.  For example, among lattices in
$\R^8$, there are $2408$ local maxima for density \cite{DSV2007}.  This
number seems to grow rapidly in high dimensions, and it means the structure
of the space of packings is complicated.  There is lots of space to move in,
with complicated geometrical configurations, and it is difficult to rule out
implausible configurations rigorously.

To get a feeling for the difficulties, it is useful to think about the
geometry of high dimensions. Let's start by looking at the $n$-dimensional
cube
\[
\{ (x_1,\dots,x_n) \ |\  |x_i| \le 1 \textup{ for all } i\}
\]
of side length $2$.  It has $2^n$ vertices $(\pm1, \dots, \pm 1)$, each at
distance $\sqrt{1^2 + \dots +1^2} = \sqrt{n}$ from the center.  When
$n=10^6$, the number of vertices is absurdly large, and they are each $1000$
units from the center, despite the fact that the side length is only $2$.
These facts are amazingly different from our intuition in low dimensions. I
like to imagine the vertices as vast numbers of tiny fingers stretching out
from the center of the cube. I find it difficult to imagine that the result
is convex, but somehow a million dimensions has enough space to accommodate
such a convex body.  The reason why cubes pack much better than spheres is
that the vertices stick out far enough to fill in all the gaps.

One of the most important insights in high-dimensional geometry is the
following principle: almost all the volume of a high-dimensional body is
concentrated near its boundary.  To see why, imagine shrinking such a body
by $1\%$, leaving just a thin fringe near the boundary. The volume of the
shrunken copy is lower by a factor of $(99/100)^n$, which tends
exponentially to zero as $n \rightarrow \infty$.  Thus, virtually all of the
volume lies in that boundary fringe. There is of course nothing special
about $1\%$. The appropriate shrinkage scale in $\R^n$ to capture a constant
fraction of the volume is on the order of $1/n$, because $(1-c/n)^n$
converges to $e^{-c}$.

Boundaries, where all the interaction takes place, become increasingly
important as dimension rises.  This helps explain the difficulty of sphere
packing, because avoiding overlap is all about interaction along boundaries.

This principle reverses our intuition from low dimensions.  We typically
think of boundaries as small and exceptional, but in high dimensions there's
practically nothing but boundary, and this changes everything.  For example,
suppose we are analyzing a numerical algorithm that uses many variables and
thus operates in a high-dimensional space.  If it works efficiently
throughout a certain region except near the boundary, then that sounds good
until we realize that almost all the region is near the boundary.

\section{Finding dense packings}

In this section we'll examine how record-setting sphere packings can be
found. The high and low-dimensional cases are handled very differently in
practice. First, we'll look at the averaging techniques used in high
dimensions, and then we'll briefly discuss how computer searches can be used
in low dimensions.

The key technique used in the most recent papers in high dimensions
\cite{Vance2011,Venkatesh2013} is the Siegel mean value theorem
\cite{Siegel1945}, which lets us average suitable functions over the space
of lattices.  To carry out such an averaging we need a probability measure,
and indeed there is a canonical probability measure on lattices with fixed
determinant (i.e., fundamental cell volume). Specifically, it's the unique
$SL_n(\R)$-invariant probability measure on this space.  The existence of an
$SL_n(\R)$-invariant measure follows from general results on Haar measure
\cite{Nachbin1965}, but it takes a calculation to show that it has finite
volume and can thus be normalized to yield a probability measure.

Once we have this probability measure on lattices, we can ask various
statistical questions.  For example, what does the average pair correlation
function look like?  In other words, what can we say about the average
number of neighbors at each distance in a random lattice?  The Siegel mean
value theorem says that these pair correlations are exactly the same as for
a Poisson distribution (i.e., uniformly scattered points).  More precisely,
it says that for a sufficiently well-behaved function $f \colon \R^n \to \R$
with $n>1$, the average of
\[
\sum_{x \in \Lambda \setminus \{0\}} f(x)
\]
over all lattices $\Lambda$ of determinant $1$ equals
\[
\int_{\R^n} f(x) \, dx.
\]
Intuitively, averaging over a random lattice blurs the sum into an integral.

The reason why the Siegel mean value theorem holds is that there is enough
symmetry to rule out any other possible answer. Specifically, by linearity
the answer must be $\int f \, d\mu$ for some measure $\mu$ on $\R^n
\setminus \{0\}$ that is invariant under $SL_n(\R)$. There is only one such
measure up to scaling when $n>1$ (given a few mild hypotheses), and some
consistency checks determine the constant of proportionality.

The meta principle here is that averaging over all possible structures is
the same as having no structure at all.  Of course this is not always true.
It generally depends on invariance under the action of a large enough group,
and $SL_n(\R)$ is more than large enough.

It is not hard to deduce lower bounds for sphere packing density from the
Siegel mean value theorem.  The following proposition is far from the state
of the art, but it illustrates the basic technique.

\begin{proposition}
The sphere packing density in $\R^n$ is at least $2 \cdot 2^{-n}$.
\end{proposition}

\begin{proof}
Let $B$ be a ball of volume $2$ centered at the origin. For a random lattice
of determinant $1$, the expected number of nonzero lattice points in $B$ is
$\vol(B) = 2$, by applying the Siegel mean value theorem to the
characteristic function of $B$.  These lattice points come in pairs
(negatives of each other), so the number is always even. Since the average
number is $2$ and some lattices have many, other lattices must have none.
Such a lattice gives a packing with one copy of $B/2$ per unit volume and
density
\[
\mathop{\mathrm{vol}}(B/2) = \frac{\mathop{\mathrm{vol}}(B)}{2^n} = 2 \cdot 2^{-n},
\]
as desired.
\end{proof}

Vance's key idea \cite{Vance2011} builds on the extra factor of $2$ that
arises because lattice vectors occur in pairs of the same length. What if we
impose additional symmetry? The intuition is that the average number of
neighbors remains the same, but now they occur in bigger clumps, and so the
chances of no nearby neighbors go up.  Vance used lattices with quaternion
algebras acting on them, and Venkatesh \cite{Venkatesh2013} obtained even
stronger results by using cyclotomic fields.

Is this the best we can do?  Only certain symmetry groups work here: we need
a big centralizer to get enough invariance for the Siegel mean value theorem
proof, and only division algebras will do.  Cyclotomic fields are the best
division algebras for this purpose \cite{Minton2011}.  Other sorts of groups
will distort the pair correlation function away from Poisson statistics, but
that could be good or bad.  The area is wide open, and it is unclear which
sorts of constructions might help.

In low dimensions, one can obtain much better results through numerical
searches by computer.  Several recent papers
\cite{KallusElserGravel2010,MarcotteTorquato2013,Kallus2013} have taken this
approach and recovered the densest lattices known in up to $20$ dimensions.
So far the computer searches have not yielded anything new, but they seem to
be on the threshold of doing so. Can we push the calculations further, to
unknown territory? What about periodic packings?

\section{Computational problems}

Lattices may sound down to earth, but they are full of computational
difficulties.  For example, given a lattice basis it is hard to tell how
dense the corresponding sphere packing is.  The difficulty is that to
compute the density, we need to know both the volume of a fundamental cell
and the packing radius of the lattice.  The former is just the absolute
value of the determinant of a basis matrix, which is easy to compute, but
computing the packing radius is not easy.  We can see why as follows.

Recall that the packing radius is half the shortest nonzero vector length in
the lattice.  The problem is that the basis vectors may not be the shortest
vectors in the lattice, because some linear combination of them could be
much shorter.  There are exponentially many linear combinations that could
work, and there is no obvious way to search efficiently.  In fact, computing
the shortest vector length is NP-hard \cite{Ajtai1998}. In other words, many
other search problems can be reduced to it. No proof is known that it cannot
be solved efficiently (this is the famous problem of whether $\textup{P} =
\textup{NP}$), but that is almost certainly the case.

There are good algorithms for ``lattice basis reduction,'' such as the LLL
algorithm \cite{LLL1982,NguyenVallee2010}, and they produce pretty short
vectors. These vectors are generally far from optimal, but they are short
enough for some applications, particularly in relatively low dimensions.

Shortest vector problems and their relatives come up in a surprising range
of topics. One beautiful application is cryptography.  We'll briefly discuss
the Goldreich-Goldwasser-Halevi cryptosystem \cite{GGH1997}, which turns out
to have weaknesses \cite{Nguyen1999} but is a good illustration of how
lattice problems can be used to build cryptosystems.  It's a public key
cryptosystem, in which the public key is a basis for a high-dimensional
lattice, while the private key is a secret nearly orthogonal basis for the
same lattice, which makes it easy to find the nearest lattice point to any
given point in space (while this problem should be hard for anyone who does
not know the secret basis). We encode messages as lattice points.  Anyone
can encrypt a message by adding a small random perturbation, thereby moving
it off the lattice.  Decryption requires finding the nearest lattice point,
which has no obvious solution without the private key. As mentioned above,
this system is not as secure as it was intended to be \cite{Nguyen1999}, but
there are other, stronger lattice-based systems. See \cite{Peikert2015} for
a survey of recent work in this area.

Recognizing algebraic numbers is a rather different sort of application. The
number
\[
\alpha = -7.82646099323767402929927644895
\]
is a 30-digit approximation to a root of a fifth-degree polynomial equation.
Which equation is it?  Of course there are infinitely many answers, but
Occam's razor suggests we should seek the simplest one.  One interpretation
of ``simplest'' is that the coefficients should be small.

For comparison,
\[
0.1345345345345345345345345345345
\]
is clearly an approximation to $1/10 + 345/9990$, and no other answer is
nearly as satisfying.

To identify the number $\alpha$ given above, let $C = 10^{20}$ (chosen based
on the precision of $\alpha$), and look at the lattice generated by the
vectors
\begin{align*}
v_0 &= (1,0,0,0,0,0,C),\\
v_1 &= (0,1,0,0,0,0,C\alpha),\\
v_2 &= (0,0,1,0,0,0,C\alpha^2),\\
v_3 &= (0,0,0,1,0,0,C\alpha^3),\\
v_4 &= (0,0,0,0,1,0,C\alpha^4),\\
v_5 &= (0,0,0,0,0,1,C\alpha^5).
\end{align*}
The lattice vectors are given by
\[
a_0v_0 + \dots + a_5v_5 = \left(a_0,a_1,a_2,a_3,a_4,a_5,C\left(\sum_{i=0}^5 a_i \alpha^i\right)\right)
\]
with $a_0,\dots,a_5 \in \Z$.   Such a vector is small when the coefficients
$a_i$ are small and the sum $\sum_{i=0}^5 a_i \alpha^i$ is tiny, since $C$
is huge.  Thus, finding a short vector amounts to finding a polynomial
$\sum_{i=0}^5 a_i x^i$ with small coefficients such that $\alpha$ is nearly
a root.

If we search for a short vector using the LLL algorithm, we find
\[
(71, \ -5, \ 12, \ -19, \ 13, \ 2, \ 0.000004135\dots).
\]
This tells us that
\[
71 - 5\alpha + 12\alpha^2 - 19\alpha^3 + 13\alpha^4 + 2\alpha^5\approx 0.
\]
(More precisely, it is about $0.000004135/C \approx 4\cdot 10^{-26}$.) In
fact, this is the equation I used to generate $\alpha$.

More generally, we can use lattices to find integral linear relations
between any real numbers, not just powers of $\alpha$.  I find it really
remarkable that the same sort of mathematics arises in this problem as in
communication over a noisy channel.

\lecture{Symmetry and ground states}

\section{Introduction}

One of the beautiful phenomena in sphere packing is the occurrence of
spontaneous order.  There seems to be no reason to expect that an optimal
sphere packing should be highly structured, but this happens time and again,
with the precise structure being difficult to predict a priori.

These questions of order vs.\ disorder fit into a broader context.  Where do
symmetry and structure come from? L\'aszl\'o Fejes T\'oth played an
important role in formulating and attracting attention to this question.  He
drew a distinction between the \emph{systematology of the regular figures},
which amounts to classifying the possible symmetries that could occur, and
the \emph{genetics of the regular figures}, which studies when and why they
do occur.  He sought to explain the genetics of the regular figures via
optimization principles, and he made considerable progress towards this
goal.  In his vision \cite[p.~x]{FejesToth1964}, ``regular arrangements are
generated from unarranged, chaotic sets by the ordering effect of an economy
principle, in the widest sense of the word.''

Typically the optimization problem has certain symmetries, but it is far
from obvious when its solutions will inherit these symmetries. Steiner trees
are an attractive illustration of this issue. What is the minimal-length
path connecting the vertices of a square?  One obvious guess is an X, which
inherits all the symmetries of the square:
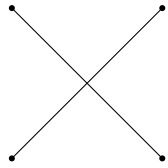
\begin{center}
\begin{tikzpicture}[x=2cm,y=2cm]
  \fill[black] (0,0) circle (0.02);
  \fill[black] (0,1) circle (0.02);
  \fill[black] (1,0) circle (0.02);
  \fill[black] (1,1) circle (0.02);
  \draw (0,0) -- (1,1);
  \draw (0,1) -- (1,0);
\end{tikzpicture}
\end{center}
However, the optimal solution turns out to look like this, or its rotation
by $90^\circ$:
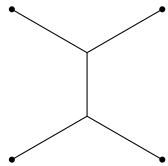
\begin{center}
\begin{tikzpicture}[x=2cm,y=2cm]
  \fill[black] (0,0) circle (0.02);
  \fill[black] (0,1) circle (0.02);
  \fill[black] (1,0) circle (0.02);
  \fill[black] (1,1) circle (0.02);
  \draw (0.5,0.288675) -- (0.5,0.711325);
  \draw (0,0) -- (0.5,0.288675);
  \draw (1,0) -- (0.5,0.288675);
  \draw (0,1) -- (0.5,0.711325);
  \draw (1,1) -- (0.5,0.711325);
\end{tikzpicture}
\end{center}
There is partial symmetry breaking, in that the set of all solutions is of
course invariant under the full symmetry group of the square, but each
individual solution is invariant under just a subgroup.

This behavior occurs generically for optimization problems.  For example, in
the sphere packing problem the full symmetry group of the optimization
problem consists of all rigid motions of Euclidean space, while each optimal
sphere packing will be invariant under a much smaller subgroup, consisting
of just a discrete set of motions. The difficulty lies in predicting what
that subgroup will be. Which materials crystallize beautifully, and which
remain amorphous?

From this perspective, we would like to understand which optimization
problems admit highly symmetrical solutions, such as lattices or regular
polytopes.  Can we explain why $E_8$ and the Leech lattice are so much more
symmetrical than the best packing known in $\R^{10}$?

\section{Potential energy minimization}

There's no hope of developing a comprehensive theory of symmetry in
optimization problems, because optimization is just too broad a topic.  If
you choose an arbitrary function to optimize, then you can make literally
anything happen to the optima.  To make progress, we must restrict the class
of functions under consideration.  In this lecture we will take a look at
point particles with pairwise forces acting on them.

Given a collection of particles interacting according to some potential
function, what do they do?  For example, the Thomson problem deals with
charged particles on the surface of the unit sphere $S^2$ in $\R^3$.  Each
pair of particles at Euclidean distance $r$ has potential energy $1/r$, and
the total potential energy is the sum over all the pairs.

The simplest question is what the \emph{ground states} are.  In other words,
what are the minimal-energy configurations?  They describe the behavior of
the system at zero temperature.  This is a simple question, but the ground
states in the Thomson problem are far from obvious, and in fact not fully
known in general.

More generally, we can ask about dynamics or the behavior at positive
temperature.  These questions are more subtle, and we will generally restrict
our attention to ground states.  After all, if we can't even understand the
ground states, then there is little hope of analyzing anything more involved
than that.

Before we restrict our attention to ground states, though, it's worth
putting everything in the context of \emph{Gibbs measures}.  They are a
canonical way of putting a probability measure on the states of a system
based on nothing except their energies and the system's temperature. Of
course one can't possibly capture the behavior of every system based on so
little information, but Gibbs measures do a good job of describing a system
that is in equilibrium with a heat bath (a neighboring system that is so
much larger that its temperature is unaffected by the smaller system).

For simplicity, imagine that our system has only $n$ states, labeled $1$
through $n$, where state $i$ has energy $E_i$.  (To handle continuous
systems we can simply replace sums over states with integrals.)  If we are
given the average energy $\overline{E}$ of the system, we determine the
corresponding probability distribution on the states by finding
probabilities $p_1,\dots,p_n$ so that $\sum_i p_i E_i = \overline{E}$ and
the \emph{entropy} $\sum_i - p_i \log p_i$ is maximized, where we interpret
$0 \log 0$ as $0$.  In other words, the system is as disordered as possible,
subject to having a certain average energy.

It is not difficult to solve this optimization problem via Lagrange
multipliers, and the result is that
\[
\log(1/p_i) = \alpha + \beta E_i
\]
for some constants $\alpha$ and $\beta$. Thus, we can write
\[
p_i = \frac{e^{-\beta E_i}}{Z},
\]
where the \emph{partition function} $Z = \sum_i e^{-\beta E_i}$ ensures that
$\sum_i p_i = 1$ (it is also $e^\alpha$).  Such a probability distribution
is called a \emph{Gibbs distribution}.

In physics terms, $\beta$ turns out to be proportional to the reciprocal of
temperature.  As the temperature tends to zero, $\beta$ tends to infinity
and the Gibbs distribution becomes concentrated on the ground states.  As
the temperature tends to infinity, $\beta$ tends to zero and the Gibbs
distribution becomes equidistributed among all the states.

One question we have not yet addressed is why $\sum_i -p_i \log p_i$
deserves the name entropy.  In fact, it is an excellent measure of disorder,
essentially because it measures how surprising the probability distribution
is on average. Consider how surprised we should be by an event of
probability $p$. Call this surprise function $S(p)$, and think of it as a
measure of how much you learn from seeing this event happen. (Information
theory makes this intuition precise.)

Clearly $S$ should be a decreasing function: the higher the probability is,
the less surprising it is and the less you learn from seeing it happen.
Furthermore, we should have
\[
S(pq) = S(p) + S(q).
\]
In other words, the amount you learn from independent events is additive.
This makes good sense intuitively: if you learn one bit of information from
a coin flip, then you learn two bits from two independent coin flips.

These conditions uniquely determine the function $S$ up to a constant
factor, as $S(p) = -\log p$.  Now the entropy is $\sum_i p_i S(p_i)$, and
this quantity measures disorder by telling us how surprised we'll be on
average by the outcome.

Part of the beauty of mathematics is that concepts are connected in ways one
would never guess.  Gibbs measures are not just a construction from
statistical physics, but rather occur throughout mathematics.  For example,
Dyson recognized that they describe eigenvalues of random matrices
\cite{Dyson1962}, as follows.

Haar measure gives a canonical probability measure on the unitary group
$U(n)$.  What does a random $n \times n$ unitary matrix chosen from this
distribution look like? It has $n$ eigenvalues $z_1,\dots,z_n$ on the unit
circle, and the Weyl integral formula tells us that the probability density
function for these eigenvalues is proportional to
\[
\prod_{i < j} |z_i-z_j|^2.
\]
If we call the constant of proportionality $1/Z$, then we can rewrite this
formula as
\[
\frac{1}{Z} e^{- 2 \sum_{i<j}\log \frac{1}{|z_i-z_j|}}.
\]
In other words, the eigenvalue distribution of a random unitary matrix is a
Gibbs distribution for a certain potential function between the eigenvalues.
Specifically, they repel each other according to the potential function $x
\mapsto \log (1/|x|)$.  This function is harmonic on $\R^2 \setminus
\{(0,0)\}$, just as the Coulomb potential $x \mapsto 1/|x|$ is harmonic on
$\R^3 \setminus \{(0,0,0)\}$. Thus, the eigenvalues of a random unitary
matrix literally repel each other via electrostatic interactions in two
dimensions, with the $2$ in the exponent specifying the temperature of this
system.

\section{Families and universal optimality}

Given that we are going to study particles interacting via pairwise
potential functions, what do we hope to learn from it?  There are many
possibilities:
\begin{enumerate}
\item We may care about the ground states for their own sake, as part of
    pure mathematics or physics (see \cite{BowickGiomi2009} for many
    examples in physics).

\item We may seek a highly uniform point distribution so that we can
    discretize the ambient space.

\item We may wish to construct error-correcting codes by letting the
    codewords repel each other, so that they become well separated.

\item We may seek well-distributed sample points for numerical
    integration.
\end{enumerate}
To account for these and other goals, we will have to look at a broad range
of potential functions.

There are also many spaces we could work in, such as spheres, projective
spaces, Grassmannians, Euclidean spaces, hyperbolic spaces, and even
discrete spaces such as the Hamming cube $\{0,1\}^n$.  All of these
possibilities are interesting, but in this lecture we will focus on spheres.
(For comparison, \cite{CohnZhao2014b} and \cite{CohnKumarMinton} examine
spaces that are rather different from spheres.)

Thus, we will focus on the question of what energy minima on spheres look
like for a variety of potential functions.  As we vary the potential
function, how do the optimal configurations change?  They vary in some
family, and we would like to understand these families.  Note that our
perspective here is broader than is typical for physics, where the potential
function is usually fixed in advance.

The simplest case is that the optimal configurations never vary, at least
for reasonable potential functions, such as inverse power laws.\footnote{Of
course it is impossible for a configuration of more than one point to be a
ground state for literally every potential function, since minimizing $f$ is
the same as maximizing $-f$.  We must restrict the class of potential
functions at least somewhat.}  For example, $4$ points on $S^2$ always form
a regular tetrahedron. Abhinav Kumar and I named this property
\emph{universal optimality} \cite{CohnKumar2007}.

More generally, we can ask for a \emph{parameter count} for the family,
which is $0$ for universal optima.  As we vary the potential function (say,
among all smooth functions), what is the dimension of the space of
configurations attained as ground states?  There is little hope of proving
much about this quantity in general.  However, we can try to estimate it
from numerical data \cite{BBCGKS2009}.  These parameter counts can be
difficult to predict, because they take into account how well the number of
points accommodates different sorts of symmetry. For example, $44$ points on
$S^2$ vary in a one-parameter family near the putative Coulomb minimizer
when we perturb the potential function, while $43$ points vary in a
$21$-parameter family. See Figure~\ref{fig:43vs44} for an illustration. What
this means is that the $44$-point configuration is nearly determined by
symmetry, with just one degree of freedom remaining to be specified by the
choice of potential function, while the $43$-point configuration is far more
complex.

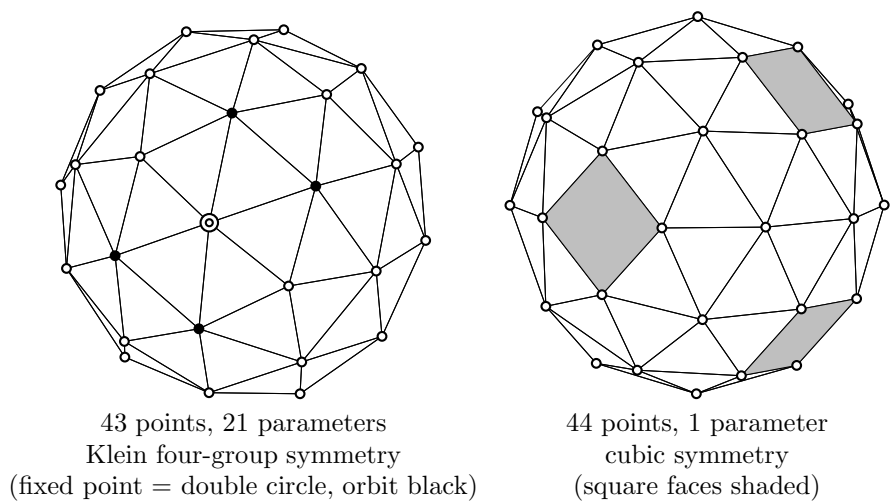
\begin{figure}
\begin{center}
\begin{tabular}{cc}
\begin{tikzpicture}[scale=2.2,x={(-0.8cm,-0.4cm)},y={(0.8cm,-0.4cm)},z={(0cm,1cm)}]
\coordinate (1) at (0.989071,0.0234765,0.145558);
\coordinate (2) at (-0.00739708,0.998991,0.0442901);
\coordinate (3) at (-0.723759,0.533184,-0.438050);
\coordinate (4) at (-0.760871,-0.567277,0.315074);
\coordinate (5) at (0.287349,-0.792474,0.537974);
\coordinate (6) at (0.236588,0.786425,0.570580);
\coordinate (7) at (0.528138,-0.848221,0.0398942);
\coordinate (8) at (-0.125044,-0.552897,0.823814);
\coordinate (9) at (-0.617477,0.276308,0.736462);
\coordinate (10) at (-0.931914,-0.0523755,0.358879);
\coordinate (11) at (0.114867,0.546171,-0.829760);
\coordinate (12) at (0.383368,-0.319466,0.866586);
\coordinate (13) at (0.0180541,-0.999506,0.0257146);
\coordinate (14) at (-0.817765,0.506753,0.272876);
\coordinate (15) at (0.381262,0.298286,0.875022);
\coordinate (16) at (-0.391710,0.383054,-0.836560);
\coordinate (17) at (-0.585394,-0.277278,0.761860);
\coordinate (18) at (-0.199586,0.850467,-0.486695);
\coordinate (19) at (-0.897891,-0.391573,-0.201154);
\coordinate (20) at (0.379951,0.823859,-0.420587);
\coordinate (21) at (0.770716,-0.494599,0.401708);
\coordinate (22) at (0.612314,-0.279720,-0.739478);
\coordinate (23) at (0.621456,0.364954,-0.693254);
\coordinate (24) at (0.221543,0.0368488,-0.974454);
\coordinate (25) at (-0.773294,-0.0290597,-0.633381);
\coordinate (26) at (-0.137758,0.494432,0.858230);
\coordinate (27) at (-0.313995,-0.849376,0.424225);
\coordinate (28) at (-0.534412,-0.835058,-0.130696);
\coordinate (29) at (0.914358,0.0187369,-0.404472);
\coordinate (30) at (-0.976923,0.161689,-0.139568);
\coordinate (31) at (-0.535343,-0.555492,-0.636268);
\coordinate (32) at (0.766147,-0.0111388,0.642569);
\coordinate (33) at (0.453914,-0.747182,-0.485470);
\coordinate (34) at (0.509680,0.854191,0.102881);
\coordinate (35) at (-0.0935043,-0.869015,-0.485870);
\coordinate (36) at (0.867555,-0.465392,-0.175383);
\coordinate (37) at (0.0805982,-0.510164,-0.856292);
\coordinate (38) at (0.844396,0.511616,-0.158885);
\coordinate (39) at (-0.352880,0.805483,0.476101);
\coordinate (40) at (-0.106782,-0.0248633,0.993972);
\coordinate (41) at (-0.309556,-0.140049,-0.940511);
\coordinate (42) at (0.749946,0.495556,0.438184);
\coordinate (43) at (-0.538315,0.841496,-0.0458407);
\draw (1)--(21);
\draw (1)--(32);
\draw (32)--(21);
\draw (1)--(21);
\draw (1)--(36);
\draw (36)--(21);
\draw (1)--(29);
\draw (1)--(36);
\draw (36)--(29);
\draw (1)--(29);
\draw (1)--(38);
\draw (38)--(29);
\draw (1)--(32);
\draw (1)--(42);
\draw (42)--(32);
\draw (1)--(38);
\draw (1)--(42);
\draw (42)--(38);
\draw (2)--(6);
\draw (2)--(34);
\draw (34)--(6);
\draw (2)--(6);
\draw (2)--(39);
\draw (39)--(6);
\draw (2)--(18);
\draw (2)--(20);
\draw (20)--(18);
\draw (2)--(18);
\draw (2)--(43);
\draw (43)--(18);
\draw (2)--(20);
\draw (2)--(34);
\draw (34)--(20);
\draw (2)--(39);
\draw (2)--(43);
\draw (43)--(39);
\draw (43)--(14);
\draw (43)--(18);
\draw (5)--(7);
\draw (5)--(7);
\draw (5)--(21);
\draw (21)--(7);
\draw (5)--(8);
\draw (5)--(12);
\draw (12)--(8);
\draw (5)--(8);
\draw (5)--(12);
\draw (5)--(21);
\draw (21)--(12);
\draw (6)--(15);
\draw (6)--(26);
\draw (26)--(15);
\draw (6)--(15);
\draw (6)--(42);
\draw (42)--(15);
\draw (6)--(26);
\draw (6)--(39);
\draw (39)--(26);
\draw (6)--(34);
\draw (6)--(42);
\draw (42)--(34);
\draw (7)--(21);
\draw (7)--(36);
\draw (36)--(21);
\draw (7)--(36);
\draw (8)--(12);
\draw (8)--(40);
\draw (40)--(12);
\draw (8)--(17);
\draw (8)--(17);
\draw (8)--(40);
\draw (40)--(17);
\draw (9)--(14);
\draw (9)--(17);
\draw (9)--(14);
\draw (9)--(39);
\draw (39)--(14);
\draw (9)--(17);
\draw (9)--(40);
\draw (40)--(17);
\draw (9)--(26);
\draw (9)--(39);
\draw (39)--(26);
\draw (9)--(26);
\draw (9)--(40);
\draw (40)--(26);
\draw (11)--(18);
\draw (11)--(18);
\draw (11)--(20);
\draw (20)--(18);
\draw (11)--(20);
\draw (11)--(23);
\draw (23)--(20);
\draw (11)--(23);
\draw (12)--(15);
\draw (12)--(32);
\draw (32)--(15);
\draw (12)--(15);
\draw (12)--(40);
\draw (40)--(15);
\draw (12)--(21);
\draw (12)--(32);
\draw (32)--(21);
\draw (14)--(39);
\draw (14)--(43);
\draw (43)--(39);
\draw (15)--(26);
\draw (15)--(40);
\draw (40)--(26);
\draw (15)--(32);
\draw (15)--(42);
\draw (42)--(32);
\draw (20)--(23);
\draw (20)--(38);
\draw (38)--(23);
\draw (20)--(34);
\draw (20)--(38);
\draw (38)--(34);
\draw (22)--(23);
\draw (22)--(23);
\draw (22)--(29);
\draw (29)--(23);
\draw (22)--(29);
\draw (22)--(36);
\draw (36)--(29);
\draw (22)--(36);
\draw (23)--(29);
\draw (23)--(38);
\draw (38)--(29);
\draw (34)--(38);
\draw (34)--(42);
\draw (42)--(38);
\fill[fill=black,draw=black,thick] (1) circle (0.75pt);
\fill[fill=white,draw=black,thick] (2) circle (0.75pt);
\fill[fill=white,draw=black,thick] (5) circle (0.75pt);
\fill[fill=black,draw=black,thick] (6) circle (0.75pt);
\fill[fill=white,draw=black,thick] (7) circle (0.75pt);
\fill[fill=white,draw=black,thick] (8) circle (0.75pt);
\fill[fill=white,draw=black,thick] (9) circle (0.75pt);
\fill[fill=white,draw=black,thick] (11) circle (0.75pt);
\fill[fill=white,draw=black,thick] (12) circle (0.75pt);
\fill[fill=white,draw=black,thick] (14) circle (0.75pt);
\fill[fill=black,draw=black,thick] (15) circle (0.75pt);
\fill[fill=white,draw=black,thick] (17) circle (0.75pt);
\fill[fill=white,draw=black,thick] (18) circle (0.75pt);
\fill[fill=white,draw=black,thick] (20) circle (0.75pt);
\fill[fill=white,draw=black,thick] (21) circle (0.75pt);
\fill[fill=white,draw=black,thick] (22) circle (0.75pt);
\fill[fill=white,draw=black,thick] (23) circle (0.75pt);
\fill[fill=white,draw=black,thick] (26) circle (0.75pt);
\fill[fill=white,draw=black,thick] (29) circle (0.75pt);
\fill[fill=white,draw=black,thick] (32) circle (0.75pt);
\fill[fill=white,draw=black,thick] (34) circle (0.75pt);
\fill[fill=white,draw=black,thick] (36) circle (0.75pt);
\fill[fill=black,draw=black,thick] (38) circle (0.75pt);
\fill[fill=white,draw=black,thick] (39) circle (0.75pt);
\fill[fill=white,draw=black,thick] (40) circle (0.75pt);
\fill[fill=white,draw=black,thick] (42) circle (1.5pt);
\fill[fill=white,draw=black,thick] (42) circle (0.6pt);
\fill[fill=white,draw=black,thick] (43) circle (0.75pt);
\end{tikzpicture}
&
\begin{tikzpicture}[scale=2.2,x={(-0.8cm,-0.4cm)},y={(0.8cm,-0.4cm)},z={(0cm,1cm)}]
\coordinate (1) at (-0.272589,-0.267387,0.924229);
\coordinate (2) at (0.680164,0.721121,-0.131760);
\coordinate (3) at (0.951792,-0.215387,0.218405);
\coordinate (4) at (-0.349122,-0.867645,-0.353986);
\coordinate (5) at (-0.181448,-0.969050,0.167388);
\coordinate (6) at (0.0863683,-0.257319,-0.962459);
\coordinate (7) at (-0.761582,-0.309941,0.569148);
\coordinate (8) at (-0.0711206,0.721301,0.688960);
\coordinate (9) at (-0.574151,0.635797,0.515861);
\coordinate (10) at (0.373911,0.708813,-0.598143);
\coordinate (11) at (-0.373911,-0.708813,0.598143);
\coordinate (12) at (-0.129120,0.623309,-0.771242);
\coordinate (13) at (-0.912733,0.229311,0.338134);
\coordinate (14) at (-0.321582,0.883546,-0.340488);
\coordinate (15) at (0.321582,-0.883546,0.340488);
\coordinate (16) at (0.272589,0.267387,-0.924229);
\coordinate (17) at (-0.759329,-0.0448494,-0.649159);
\coordinate (18) at (0.129120,-0.623309,0.771242);
\coordinate (19) at (-0.951792,0.215387,-0.218405);
\coordinate (20) at (0.349122,0.867645,0.353986);
\coordinate (21) at (0.759329,0.0448494,0.649159);
\coordinate (22) at (0.407575,0.453734,0.792469);
\coordinate (23) at (-0.967039,-0.248595,0.0550935);
\coordinate (24) at (0.574151,-0.635797,-0.515861);
\coordinate (25) at (-0.589399,0.171814,0.789360);
\coordinate (26) at (-0.0863683,0.257319,0.962459);
\coordinate (27) at (0.967039,0.248595,-0.0550935);
\coordinate (28) at (0.728905,-0.684404,0.0169775);
\coordinate (29) at (-0.407575,-0.453734,-0.792469);
\coordinate (30) at (-0.774577,-0.508832,-0.375661);
\coordinate (31) at (0.316445,-0.126700,0.940111);
\coordinate (32) at (0.181448,0.969050,-0.167388);
\coordinate (33) at (0.0711206,-0.721301,-0.688960);
\coordinate (34) at (-0.728905,0.684404,-0.0169775);
\coordinate (35) at (-0.680164,-0.721121,0.131760);
\coordinate (36) at (-0.316445,0.126700,-0.940111);
\coordinate (37) at (0.761582,0.309941,-0.569148);
\coordinate (38) at (0.640144,-0.496698,0.586095);
\coordinate (39) at (0.912733,-0.229311,-0.338134);
\coordinate (40) at (0.589399,-0.171814,-0.789360);
\coordinate (41) at (0.232569,-0.950432,-0.206375);
\coordinate (42) at (-0.232569,0.950432,0.206375);
\coordinate (43) at (-0.640144,0.496698,-0.586095);
\coordinate (44) at (0.774577,0.508832,0.375661);
\draw (1)--(25);
\draw (1)--(18);
\draw (1)--(18);
\draw (1)--(31);
\draw (31)--(18);
\draw (1)--(25);
\draw (1)--(26);
\draw (26)--(25);
\draw (1)--(26);
\draw (1)--(31);
\draw (31)--(26);
\draw (2)--(10);
\draw (2)--(32);
\draw (32)--(10);
\draw (2)--(10);
\draw (2)--(37);
\draw (37)--(10);
\draw (2)--(20);
\draw (2)--(32);
\draw (32)--(20);
\draw (2)--(20);
\draw (2)--(44);
\draw (44)--(20);
\draw (2)--(27);
\draw (2)--(37);
\draw (37)--(27);
\draw (2)--(27);
\draw (2)--(44);
\draw (44)--(27);
\draw (3)--(21);
\draw (3)--(38);
\draw (38)--(21);
\draw (3)--(27);
\draw (3)--(39);
\draw (39)--(27);
\draw (3)--(28);
\draw (3)--(38);
\draw (38)--(28);
\draw (3)--(28);
\draw (3)--(39);
\draw (39)--(28);
\draw (40)--(16);
\draw (25)--(13);
\draw (8)--(9);
\draw (8)--(42);
\draw (42)--(9);
\draw (8)--(20);
\draw (8)--(22);
\draw (22)--(20);
\draw (8)--(20);
\draw (8)--(42);
\draw (42)--(20);
\draw (8)--(22);
\draw (8)--(26);
\draw (26)--(22);
\draw (9)--(13);
\draw (9)--(25);
\draw (25)--(13);
\draw (9)--(13);
\draw (9)--(34);
\draw (34)--(13);
\draw (9)--(34);
\draw (9)--(42);
\draw (42)--(34);
\draw (10)--(12);
\draw (10)--(16);
\draw (16)--(12);
\draw (10)--(16);
\draw (10)--(37);
\draw (37)--(16);
\draw (12)--(14);
\draw (12)--(16);
\draw (13)--(34);
\draw (14)--(32);
\draw (14)--(42);
\draw (42)--(32);
\draw (14)--(34);
\draw (14)--(42);
\draw (42)--(34);
\draw (14)--(34);
\draw (15)--(18);
\draw (15)--(38);
\draw (38)--(18);
\draw (15)--(28);
\draw (15)--(38);
\draw (38)--(28);
\draw (15)--(28);
\draw (16)--(37);
\draw (16)--(40);
\draw (40)--(37);
\draw (18)--(31);
\draw (18)--(38);
\draw (38)--(31);
\draw (20)--(22);
\draw (20)--(44);
\draw (44)--(22);
\draw (20)--(32);
\draw (20)--(42);
\draw (42)--(32);
\draw (21)--(22);
\draw (21)--(31);
\draw (31)--(22);
\draw (21)--(22);
\draw (21)--(44);
\draw (44)--(22);
\draw (21)--(31);
\draw (21)--(38);
\draw (38)--(31);
\draw (22)--(26);
\draw (22)--(31);
\draw (31)--(26);
\draw (39)--(28);
\draw (40)--(39);
\draw (27)--(37);
\draw (27)--(39);
\draw (39)--(37);
\draw (37)--(39);
\draw (37)--(40);
\draw (40)--(39);
\draw[fill=black!25] (3)--(27)--(44)--(21)--(3);
\draw[fill=black!25] (8)--(9)--(25)--(26)--(8);
\draw[fill=black!25] (10)--(12)--(14)--(32)--(10);
\fill[fill=white,draw=black,thick] (1) circle (0.75pt);
\fill[fill=white,draw=black,thick] (2) circle (0.75pt);
\fill[fill=white,draw=black,thick] (3) circle (0.75pt);
\fill[fill=white,draw=black,thick] (8) circle (0.75pt);
\fill[fill=white,draw=black,thick] (9) circle (0.75pt);
\fill[fill=white,draw=black,thick] (10) circle (0.75pt);
\fill[fill=white,draw=black,thick] (12) circle (0.75pt);
\fill[fill=white,draw=black,thick] (13) circle (0.75pt);
\fill[fill=white,draw=black,thick] (14) circle (0.75pt);
\fill[fill=white,draw=black,thick] (15) circle (0.75pt);
\fill[fill=white,draw=black,thick] (16) circle (0.75pt);
\fill[fill=white,draw=black,thick] (18) circle (0.75pt);
\fill[fill=white,draw=black,thick] (20) circle (0.75pt);
\fill[fill=white,draw=black,thick] (21) circle (0.75pt);
\fill[fill=white,draw=black,thick] (22) circle (0.75pt);
\fill[fill=white,draw=black,thick] (25) circle (0.75pt);
\fill[fill=white,draw=black,thick] (26) circle (0.75pt);
\fill[fill=white,draw=black,thick] (27) circle (0.75pt);
\fill[fill=white,draw=black,thick] (28) circle (0.75pt);
\fill[fill=white,draw=black,thick] (31) circle (0.75pt);
\fill[fill=white,draw=black,thick] (32) circle (0.75pt);
\fill[fill=white,draw=black,thick] (34) circle (0.75pt);
\fill[fill=white,draw=black,thick] (37) circle (0.75pt);
\fill[fill=white,draw=black,thick] (38) circle (0.75pt);
\fill[fill=white,draw=black,thick] (39) circle (0.75pt);
\fill[fill=white,draw=black,thick] (40) circle (0.75pt);
\fill[fill=white,draw=black,thick] (42) circle (0.75pt);
\fill[fill=white,draw=black,thick] (44) circle (0.75pt);
\end{tikzpicture}\\
$43$ points, $21$ parameters& $44$ points, $1$ parameter\\
Klein four-group symmetry & cubic symmetry\\
(fixed point = double circle, orbit black) & (square faces shaded)
\end{tabular}
\end{center}
\caption{Putative ground states for Coulomb energy on $S^2$, and the number
of parameters for the families they lie in.} \label{fig:43vs44}
\end{figure}

To give a precise definition of universal optimality, we must specify the
class of potential functions.  For a finite subset $\mathcal{C} \subset
S^{n-1}$ and a function $f \colon (0,4] \to \R$, we define the \emph{energy
of $\mathcal{C}$ with respect to the potential function $f$} to be
\[
E_f(\mathcal{C}) = \frac{1}{2}
\sum_{\substack{x,y \in \mathcal{C}\\ x \ne y}}
f\big(|x-y|^2\big).
\]
The factor of $1/2$ simply corrects for counting each pair twice and is not
important. The use of squared Euclidean distance similarly doesn't matter in
principle, since the squaring could be incorporated into the potential
function, but it turns out to be a surprisingly useful convention.

A function $f$ is \emph{completely monotonic} if it is infinitely
differentiable and
\[
(-1)^k f^{(k)} \ge 0
\]
for all $k \ge 0$ (i.e., its derivatives alternate in sign, as in inverse
power laws).  We say $\mathcal{C}$ is \emph{universally optimal} if it
minimizes $E_f(\mathcal{C})$ for all completely monotonic $f$, compared with
all $|\mathcal{C}|$-point configurations on $S^{n-1}$.

It's not obvious that completely monotonic functions are the right class of
functions to use, but they turn out to be.  The fact that $f$ is decreasing
means the force is repulsive, and convexity means the force grows stronger
at short distances.  Complete monotonicity is a natural generalization of
these conditions, and the results and examples in \cite{CohnKumar2007} give
evidence that it is the right generalization (see pages 101 and 107--108).
Note in particular that inverse power laws are completely monotonic, so
universal optima must minimize energy for all inverse power laws.

In the circle $S^1$, there is a universal optimum of each size, namely the
regular polygon.  This is not as straightforward to prove as it sounds, but
it follows from Theorem~1.2 in \cite{CohnKumar2007}, which we will state as
Theorem~\ref{thm:cohnkumar2007} in the fourth lecture. In $S^2$, the
complete list of universal optima with more than one point is as follows:
\begin{enumerate}
\item Two antipodal points ($2$ points)

\item Equilateral triangle on equator ($3$ points)

\item Regular tetrahedron ($4$ points)

\item Regular octahedron ($6$ points)

\item Regular icosahedron ($12$ points)
\end{enumerate}
See Figure~\ref{fig:platonic}. Universal optimality again follows from
Theorem~1.2 in \cite{CohnKumar2007} (after special cases were proved in
\cite{Yudin1993,KolushovYudin1994,Andreev1996,KolushovYudin1997,Andreev1997}),
while completeness follows from a theorem of Leech in \cite{Leech1957}.

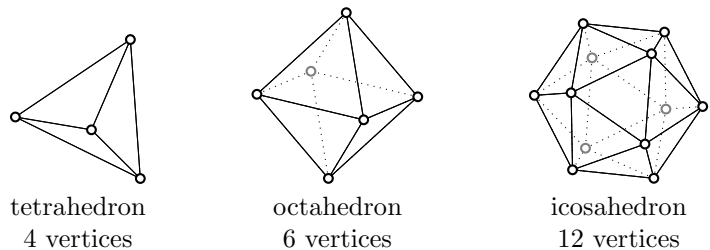
\begin{figure}
\begin{center}
\begin{tabular}{ccccc}
\begin{tikzpicture}[scale=1,x={(-0.8cm,-0.4cm)},y={(0.8cm,-0.4cm)},z={(0cm,1cm)}]
\coordinate (1) at (-0.719912,-0.121578,0.683334);
\coordinate (2) at (0.430908,-0.882004,-0.190756);
\coordinate (3) at (0.686543,0.637867,0.348977);
\coordinate (4) at (-0.397539,0.365715,-0.841555);
\draw (1)--(2);
\draw (1)--(3);
\draw (3)--(2);
\draw (1)--(2);
\draw (1)--(4);
\draw (4)--(2);
\draw (1)--(3);
\draw (1)--(4);
\draw (4)--(3);
\draw (2)--(3);
\draw (2)--(4);
\draw (4)--(3);
\fill[fill=white,draw=black,thick] (1) circle (1.65pt);
\fill[fill=white,draw=black,thick] (2) circle (1.65pt);
\fill[fill=white,draw=black,thick] (3) circle (1.65pt);
\fill[fill=white,draw=black,thick] (4) circle (1.65pt);
\end{tikzpicture}
&
\qquad \quad &
\begin{tikzpicture}[scale=1,x={(-0.8cm,-0.4cm)},y={(0.8cm,-0.4cm)},z={(0cm,1cm)}]
\coordinate (1) at (-0.462126,0.874345,0.148193);
\coordinate (2) at (-0.568973,-0.420506,0.706714);
\coordinate (3) at (0.680228,0.242273,0.691805);
\coordinate (4) at (0.462126,-0.874345,-0.148193);
\coordinate (5) at (-0.680228,-0.242273,-0.691805);
\coordinate (6) at (0.568973,0.420506,-0.706714);
\draw[dotted] (1)--(2);
\draw[dotted] (1)--(3);
\draw[dotted] (2)--(3);
\draw[dotted] (1)--(5);
\draw[dotted] (2)--(5);
\draw[dotted] (1)--(6);
\draw[dotted] (3)--(6);
\draw[dotted] (5)--(6);
\draw[dotted] (2)--(4);
\draw[dotted] (3)--(4);
\draw[dotted] (4)--(5);
\draw[dotted] (4)--(6);
\draw (1)--(2);
\draw (1)--(2);
\draw (1)--(5);
\draw (5)--(2);
\draw (1)--(6);
\draw (1)--(5);
\draw (1)--(6);
\draw (6)--(5);
\draw (2)--(4);
\draw (2)--(4);
\draw (2)--(5);
\draw (5)--(4);
\draw (6)--(4);
\draw (4)--(5);
\draw (4)--(6);
\draw (6)--(5);
\fill[fill=white,draw=black,thick] (1) circle (1.65pt);
\fill[fill=white,draw=black,thick] (2) circle (1.65pt);
\fill[fill=white,draw=black!50,thick] (3) circle (1.65pt);
\fill[fill=white,draw=black,thick] (4) circle (1.65pt);
\fill[fill=white,draw=black,thick] (5) circle (1.65pt);
\fill[fill=white,draw=black,thick] (6) circle (1.65pt);
\end{tikzpicture}
&
\qquad \quad &
\begin{tikzpicture}[scale=1,x={(-0.8cm,-0.4cm)},y={(0.8cm,-0.4cm)},z={(0cm,1cm)}]
\coordinate (1) at (0.451735,0.891412,-0.0363343);
\coordinate (2) at (-0.451735,-0.891412,0.0363343);
\coordinate (3) at (0.595066,-0.803624,-0.00920844);
\coordinate (4) at (-0.595066,0.803624,0.00920844);
\coordinate (5) at (0.0798685,-0.507188,0.858127);
\coordinate (6) at (-0.868867,-0.0812855,-0.488326);
\coordinate (7) at (-0.00871491,0.540402,0.841362);
\coordinate (8) at (0.868867,0.0812855,0.488326);
\coordinate (9) at (-0.824892,-0.0607582,0.562016);
\coordinate (10) at (-0.0798685,0.507188,-0.858127);
\coordinate (11) at (0.824892,0.0607582,-0.562016);
\coordinate (12) at (0.00871491,-0.540402,-0.841362);
\draw[dotted] (1)--(4);
\draw[dotted] (1)--(7);
\draw[dotted] (4)--(7);
\draw[dotted] (1)--(10);
\draw[dotted] (4)--(10);
\draw[dotted] (1)--(8);
\draw[dotted] (7)--(8);
\draw[dotted] (1)--(11);
\draw[dotted] (8)--(11);
\draw[dotted] (10)--(11);
\draw[dotted] (2)--(3);
\draw[dotted] (2)--(5);
\draw[dotted] (3)--(5);
\draw[dotted] (2)--(12);
\draw[dotted] (3)--(12);
\draw[dotted] (2)--(9);
\draw[dotted] (5)--(9);
\draw[dotted] (2)--(6);
\draw[dotted] (6)--(9);
\draw[dotted] (6)--(12);
\draw[dotted] (3)--(8);
\draw[dotted] (5)--(8);
\draw[dotted] (3)--(11);
\draw[dotted] (11)--(12);
\draw[dotted] (4)--(6);
\draw[dotted] (4)--(9);
\draw[dotted] (6)--(10);
\draw[dotted] (7)--(9);
\draw[dotted] (5)--(7);
\draw[dotted] (10)--(12);
\draw (1)--(4);
\draw (1)--(7);
\draw (7)--(4);
\draw (1)--(4);
\draw (1)--(10);
\draw (10)--(4);
\draw (1)--(7);
\draw (1)--(8);
\draw (8)--(7);
\draw (1)--(8);
\draw (1)--(11);
\draw (11)--(8);
\draw (1)--(10);
\draw (1)--(11);
\draw (11)--(10);
\draw (5)--(3);
\draw (9)--(5);
\draw (3)--(5);
\draw (3)--(8);
\draw (8)--(5);
\draw (3)--(8);
\draw (3)--(11);
\draw (11)--(8);
\draw (3)--(11);
\draw (4)--(9);
\draw (4)--(10);
\draw (4)--(7);
\draw (4)--(9);
\draw (9)--(7);
\draw (5)--(7);
\draw (5)--(8);
\draw (8)--(7);
\draw (5)--(7);
\draw (5)--(9);
\draw (9)--(7);
\draw (10)--(11);
\fill[fill=white,draw=black,thick] (1) circle (1.65pt);
\fill[fill=white,draw=black!50,thick] (2) circle (1.65pt);
\fill[fill=white,draw=black,thick] (3) circle (1.65pt);
\fill[fill=white,draw=black,thick] (4) circle (1.65pt);
\fill[fill=white,draw=black,thick] (5) circle (1.65pt);
\fill[fill=white,draw=black!50,thick] (6) circle (1.65pt);
\fill[fill=white,draw=black,thick] (7) circle (1.65pt);
\fill[fill=white,draw=black,thick] (8) circle (1.65pt);
\fill[fill=white,draw=black,thick] (9) circle (1.65pt);
\fill[fill=white,draw=black,thick] (10) circle (1.65pt);
\fill[fill=white,draw=black,thick] (11) circle (1.65pt);
\fill[fill=white,draw=black!50,thick] (12) circle (1.65pt);
\end{tikzpicture}\\
tetrahedron & & octahedron & & icosahedron\\
4 vertices & & 6 vertices & & 12 vertices
\end{tabular}
\end{center}
\caption{Platonic solids whose vertices form universally optimal codes.}
\label{fig:platonic}
\end{figure}

The cube and regular dodecahedron are conspicuously missing from this list.
The cube cannot be universally optimal, because rotating one face moves its
corners further from those of the opposite face, and the dodecahedron fails
similarly.  Square and pentagonal faces are not particularly favorable
shapes for energy minimization, although cubes and dodecahedra can occur for
unusual potential functions \cite{CohnKumar2009b}.

Five points are the first case without universal optimality, and they are
surprisingly subtle. There are two natural ways to arrange the particles: we
could include the north and south poles together with an equilateral
triangle on the equator (a \emph{triangular bipyramid}), or the north pole
together with a square at constant latitude in the southern hemisphere (a
\emph{square pyramid}). The square pyramid lies in a one-parameter family,
where the latitude of the square depends on the choice of potential
function.  By contrast, the triangular bipyramid is in equilibrium for every
potential function, but it becomes an unstable equilibrium for steep inverse
power laws.

\begin{conjecture}
For each completely monotonic potential function, either the triangular
bipyramid or a square pyramid minimizes energy for $5$ points in $S^2$.
\end{conjecture}

This conjecture really feels like it ought to be provable.  Specifying five
points on $S^2$ takes ten degrees of freedom, three of which are lost if we
take the quotient by symmetries.  Thus, we are faced with a calculus problem
in just seven variables.  However, despite a number of partial results
\cite{DLT2002,Schwartz2013,BHS2014,Schwartz2015}, no complete solution is
known.

\begin{table}
\caption{Known universal optima with $N$ points on $S^{n-1}$.}
\label{tab:univopt}
\begin{center}
\begin{tabular}{ccc}
\toprule
$n$ & $N$ & Description\\
\midrule
{$2$} & {$N$} & {$N$-gon}\\
{$n$} & {$N \le n+1$} & {simplex} (generalized tetrahedron)\\
{$n$} & {$2n$} & {cross polytope} (generalized octahedron)\\
{$3$} & {$12$} & {icosahedron}\\
{$4$} & {$120$} & {regular $600$-cell}\\
$5$ & $16$ & hemicube\\
{$6$} & {$27$} & {Schl\"afli graph}\\
$7$ & $56$ & equiangular lines\\
{$8$} & {$240$} & {$E_8$ root system}\\
$21$ & $112$ & isotropic subspaces\\
$21$ & $162$  & strongly regular graph\\
$22$ & $100$  & Higman-Sims graph\\
$22$ & $275$  & McLaughlin graph\\
$22$ & $891$  & isotropic subspaces\\
$23$ & $552$  & equiangular lines\\
$23$ & $4600$  & kissing configuration of next line\\
$24$ & $196560$ & Leech lattice minimal vectors\\
$q(q^3+1)/(q+1)$ & $(q+1)(q^3+1)$ & isotropic subspaces ($q$ is a
prime power)\\
\bottomrule
\end{tabular}
\end{center}
\end{table}

The known universal optima in spheres are listed in Table~\ref{tab:univopt}.
Each of them is an exciting mathematical object that predates the study of
universal optimality. For example, the $27$ points in $\R^6$ correspond to
the classical configuration of $27$ lines on a cubic surface.  One way of
thinking about universal optimality is that it highlights similarities
between various exceptional structure and helps characterize what's so
special about them.  See \cite{CohnKumar2007} for descriptions of these
objects and how they are related. We'll discuss the proof techniques in the
fourth lecture, while \cite{CohnKumar2007} contains detailed proofs.

One important source of universal optima is regular polytopes, the
higher-dimensional generalizations of the Platonic solids.  As in three
dimensions, only some of them are universally optimal, specifically the ones
with simplicial facets.\footnote{Surprisingly, the minimal vectors of $D_4$
are not universally optimal \cite{CCEK2007}, despite their beauty and
symmetry. They are the vertices of a regular polytope with octahedral
facets, called the regular $24$-cell.} The shortest vectors in the $E_8$
lattice (called the $E_8$ root system) also form a universally optimal
configuration, as do the shortest vectors in the Leech lattice.

\begin{figure}
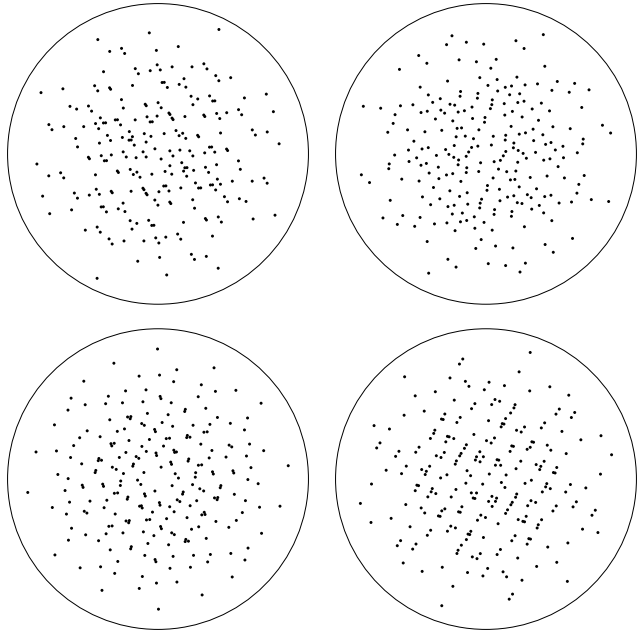

\begin{center}

\end{center}
\caption{Four views of the $E_8$ root system, after orthogonal projection onto randomly
chosen planes.}
\label{fig:E8}
\end{figure}

\begin{figure}
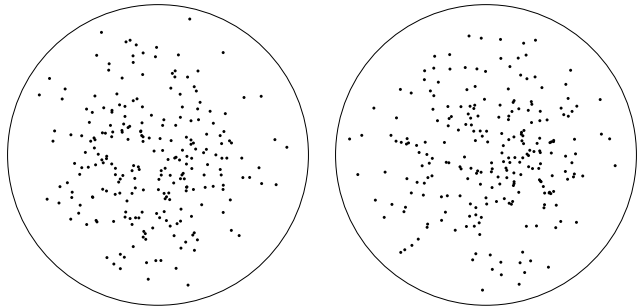

\begin{center}
%
\end{center}
\caption{A random $240$-point configuration in $S^7$, orthogonally projected onto randomly chosen planes.}
\label{fig:random}
\end{figure}

It is difficult to depict high-dimensional objects on a two-dimensional
page, but Figure~\ref{fig:E8} shows how the $E_8$ root system appears when
viewed from random directions.  It is so regular and symmetrical that even
these random views display considerable structure.  For comparison,
Figure~\ref{fig:random} shows similar projections of a random point
configuration.

In up to $24$ dimensions, all of the known universal optima are regular
polytopes or cross sections of the $E_8$ or Leech configurations.  However,
the last line of Table~\ref{tab:univopt} shows that there are more examples
coming from finite geometry.  It's not plausible that
Table~\ref{tab:univopt} is the complete list of universal optima, and in
fact \cite{BBCGKS2009} constructs two conjectural examples ($40$ points in
$\R^{10}$ and $64$ points in $\R^{14}$), but it seems difficult to find or
analyze further universal optima.

The gap between $8$ and $21$ dimensions in Table~\ref{tab:univopt} is
puzzling.  Are the dimensions in between not favored by universal
optimality, or do we just lack the imagination to construct new universal
optima in these dimensions?

\section{Optimality of simplices}

It is not difficult to explore energy minimization via numerical
optimization, but it is far from obvious how to prove anything about it.
Developing proof techniques will occupy most of the remaining lectures, and
we will start here by analyzing regular simplices, i.e., configurations of
equidistant points.

In particular, we will study the \emph{spherical code} problem: how can we
maximize the closest distance between $N$ points on the unit sphere
$S^{n-1}$?  This is an important problem in both geometry and information
theory.\footnote{If we represent radio signals by vectors in $\R^n$ by
measuring the amplitudes at different frequencies, then the squared vector
length is proportional to the power of the radio signal.  If we transmit a
constant-power signal, then we need an error-correcting code on the surface
of a sphere, i.e., a spherical code.}  It is a version of the sphere packing
problem in spherical geometry, i.e., for spherical caps on the surface of a
sphere. Furthermore, it is a degenerate case of energy minimization.  If we
look at the limit of increasingly steep potential functions, then
asymptotically only the minimal distance matters and we obtain an optimal
spherical code.

When $N \le n+1$, we will see shortly that the optimal solution is a
\emph{regular simplex}.  In other words, the points are all equidistant from
each other, forming an $n$-dimensional analogue of the equilateral triangle
or regular tetrahedron. The cutoff at $n+1$ simply reflects the fact that
$\R^n$ cannot contain more than $n+1$ equidistant points.

Let $\langle x,y \rangle$ denote the inner product of $x$ and $y$.  Inner
products  can be used to measure distances on the unit sphere, since
\[
|x-y|^2 = \langle x-y,x-y \rangle = |x|^2+|y|^2 - 2\langle x,y \rangle = 2 - 2 \langle x,y \rangle
\]
when $|x|=|y|=1$.  Thus, maximizing the distance $|x-y|$ is equivalent to
minimizing the inner product $\langle x,y \rangle$.

Note that if $x_1,\dots,x_N$ are unit vectors forming the vertices of a
regular simplex centered at the origin, then all the inner products between
them must be $-1/(N-1)$.  To see why, observe that $x_1+\dots+x_N = 0$ and
hence
\[
0 = |x_1+\dots+x_N|^2 = N + \sum_{i\ne j} \langle x_i,x_j \rangle,
\]
while all the $N(N-1)$ inner products in this sum are equal.  This
calculation already contains all the ingredients needed to prove that
regular simplices are optimal spherical codes:

\begin{proposition} \label{prop:simplex}
If $N \le n+1$, then the unique optimal $N$-point spherical code in
$S^{n-1}$ is the regular simplex centered at the origin.
\end{proposition}

Of course it is unique only up to rigid motions.

\begin{proof}
Suppose $x_1,\dots,x_N$ are points on $S^{n-1}$.  The fundamental inequality
we'll use is
\[
\left|\sum_{i=1}^N x_i \right|^2 \ge 0.
\]
Using $|x_i|^2=1$, this inequality expands to
\[
N + \sum_{i \ne j} \langle x_i, x_j \rangle \ge 0,
\]
which amounts to
\[
\frac{1}{N(N-1)}\sum_{i \ne j} \langle x_i,x_j \rangle \ge -\frac{1}{N-1}.
\]
In other words, the average inner product is at least $-1/(N-1)$, and hence
the maximal inner product (which corresponds to the minimal distance) must
be at least that large. Equality holds iff all the inner products are the
same and $\sum_i x_i = 0$.  This condition is equivalent to all the points
being equidistant with centroid at the origin, which can be achieved iff $N
\le n+1$.
\end{proof}

\begin{exercise} \label{ex:simplicesunivopt}
Prove that regular simplices are universally optimal, and more generally
that they minimize $E_f$ for every decreasing, convex potential function
$f$.
\end{exercise}

Our discussion here may give the impression that the existence of regular
simplices is trivial, while their optimality is a little more subtle.  This
impression is reasonable for Euclidean space, but in projective spaces or
Grassmannians the existence of regular simplices is far more mysterious.
See, for example, \cite{CohnKumarMinton}.

The inequality
\[
\left|\sum_{i=1}^N x_i \right|^2 \ge 0
\]
is useful for analyzing simplices, but it is not obvious at a glance what
its significance is or how it fits into a broader theory. It turns out to be
a special case of Delsarte's linear programming bounds, which are also
equivalent to the nonnegativity of the structure factor in statistical
physics. In the upcoming lectures, we'll look at these connections.  The
fundamental theme will be geometrical constraints on correlation functions.

\lecture{Interlude: Spherical harmonics}

Spherical harmonics are a spherical generalization of Fourier series and a
fundamental tool for understanding particle configurations on the surface of
a sphere. Despite their importance in mathematics, they are not nearly as
well known as Fourier series are, so this lecture will be devoted to the
basic theory. We'll begin with a quick review of Fourier series, to
establish notation and fundamental concepts, and then we'll do the same
things in higher dimensions.  Our discussion will start off in a rather
elementary fashion, but then gradually increase in sophistication. We won't
go through complete proofs of basic facts such as convergence of Fourier
series under the $L^2$ norm, but we will at least see an outline of what is
true and why, to a level of detail at which the proofs could be completed
using standard facts from introductory graduate classes.

\section{Fourier series}
\label{sec:fourierseries}

We will identify the circle $S^1$ with the quotient $\R/2\pi\Z$ via arc
length (i.e., the quotient of the real line in which we wrap around after
$2\pi$ units). In other words, a function on the circle is the same as a
function on $\R$ with period $2\pi$.

We know from basic analysis that every sufficiently nice function $f$ from
$\R/2\pi\Z$ to $\C$ can be expanded in a Fourier series
\begin{equation} \label{eq:fourierseries}
f(x) = \sum_{k \in \Z} a_k e^{i k x}.
\end{equation}
Of course we could replace the complex exponentials with trigonometric
functions by writing $e^{ikx} = \cos kx + i \sin kx$, but the exponentials
will be more pleasant.

The coefficients $a_\ell$ are determined by orthogonality via
\[
a_\ell = \frac{1}{2\pi}\int_0^{2\pi} f(x) e^{-i\ell x} \, dx,
\]
because we can interchange the sum \eqref{eq:fourierseries} with the
integral and apply
\begin{equation} \label{eq:orthonormal}
\frac{1}{2\pi}\int_0^{2\pi} e^{i(k-\ell)x} \, dx = \begin{cases} 1 & \textup{if $k=\ell$, and}\\
0 & \textup{otherwise.}
\end{cases}
\end{equation}

The right setting for Fourier series is the space of square-integrable
functions on $S^1$, i.e.,
\[
L^2(S^1) = \left\{f \colon \R/2\pi\Z \to \C \ \Big|\  \int_0^{2\pi} |f(x)|^2 \, dx < \infty \right\}.
\]
This is a Hilbert space under the inner product $\langle \cdot,\cdot \rangle$
defined by
\[
\langle f,g \rangle = \frac{1}{2\pi}\int_0^{2\pi} \overline{f(x)} g(x) \, dx,
\]
which corresponds to the $L^2$ norm $\|\cdot\|_2$ defined by
\[
\|f\|_2 = \sqrt{\langle f,f \rangle} = \sqrt{\frac{1}{2\pi} \int_0^{2\pi} |f(x)|^2 \, dx}.
\]
The exponential functions are orthonormal in $L^2(S^1)$: if $f_k$ is the
function defined by $f_k(x) = e^{i k x}$, then \eqref{eq:orthonormal}
amounts to
\[
\langle f_k, f_\ell \rangle = \begin{cases} 1 & \textup{if $k=\ell$, and}\\
0 & \textup{otherwise.}
\end{cases}
\]
Furthermore, these functions form an orthonormal basis of $L^2(S^1)$.

We can express this fact algebraically as follows. If $V_k$ consists of the
complex multiples of the function $f_k$, then
\[
L^2(S^1) = \hatbigoplus_{k \in \Z} V_k.
\]
(Here $\oplus$ is the orthogonal direct sum.  The hat indicates a Hilbert
space completion; without the hat, the direct sum would contain only sums of
finitely many exponentials.) In other words, the partial sums of the Fourier
series of an $L^2$ function converge to that function under the $L^2$ norm.
However, it's important to keep in mind that they needn't converge pointwise.

The most important property of the decomposition
\[
L^2(S^1) = \hatbigoplus_{k \in \Z} V_k.
\]
is that it is compatible with the symmetries of $S^1$ (i.e., the rigid
motions that preserve $S^1$), as we will see shortly. Recall that the
symmetry group $O(2)$ of $S^1$ consists of rotations and reflections that
fix the center of the circle, with the subgroup $SO(2)$ consisting of just
the rotations. The notation is based on the fact that these symmetries can
be written in terms of orthogonal matrices, but we do not need that
perspective here.

Each symmetry $g$ of $S^1$ acts on functions $f \colon S^1 \to \C$ by
sending $f$ to the function $gf$ defined by $(gf)(x) = f(g^{-1}x)$.  The
inverse ensures that the associative law $(gh)f = g(hf)$ holds.  For
motivation, recall that moving the graph of a function $f(x)$ one unit to
the right amounts to graphing $f(x-1)$, not $f(x+1)$.  Similarly, the graph
of $gf$ is simply the graph of $f$ transformed according to $g$.

Under this action, $L^2(S^1)$ is a representation of the group $O(2)$.  In
other words, the group $O(2)$ acts on $L^2(S^1)$ by linear transformations.
In fact, it is a unitary representation, which means that symmetries of
$S^1$ preserve the $L^2$ norm.  We would like to decompose $L^2(S^1)$ into
irreducible representations of $O(2)$ or $SO(2)$. In other words, we would
like to break it apart into orthogonal subspaces preserved by these groups,
with the subspaces being as small as possible.

For the rotation group $SO(2)$, we're already done.  In the $\R/2\pi\Z$
picture, rotations of $S^1$ correspond to translations of $\R$.  The
exponential functions are already invariant: if we translate $x \mapsto
e^{ikx}$ by $t$, we get
\[
e^{ik(x-t)} = e^{-ikt} e^{ikx},
\]
which is the original function $x \mapsto e^{ikx}$ multiplied by the
constant $e^{-ikt}$.  In other words, $V_k$ is itself a representation of
$SO(2)$, and
\[
L^2(S^1) = \hatbigoplus_{k \in \Z} V_k
\]
is the complete decomposition of $L^2(S^1)$ under this group action.  Each
summand must be irreducible, since it's one-dimensional.

There are many ways to restate this decomposition, such as:
\begin{enumerate}
\item The Fourier basis simultaneously diagonalizes the translation
    operators on $L^2(\R/2\pi\Z)$ (i.e., rotations of $L^2(S^1)$).

\item The exponential functions are simultaneous eigenfunctions for the
    translation operators.
\end{enumerate}
It turns out that the reason why the Fourier decomposition is particularly
simple, with one-dimensional summands, is that the rotation group $SO(2)$ is
abelian.

But what about the full symmetry group $O(2)$?  It is generated by $SO(2)$
and any one reflection, because all the reflections are conjugate by
rotations.  In the $\R/2\pi\Z$ picture, we can use the reflection $x \mapsto
-x$.  The nonconstant exponential functions are not preserved by this
reflection, because it takes $x \mapsto e^{ikx}$ to $x \mapsto e^{-ikx}$. In
other words, it interchanges $k$ with $-k$.

However, this is no big deal.  Instead of keeping the representations $V_k$
and $V_{-k}$ separate, we combine them to form $W_k = V_k \oplus V_{-k}$
when $k>0$ (while we take $W_0=V_0$). Now $W_k$ is the span of $x \mapsto
e^{ikx}$ and $x \mapsto e^{-ikx}$, or equivalently $x \mapsto \cos kx$ and
$x \mapsto \sin kx$ if we expand $e^{\pm ikx} = \cos kx \pm i \sin kx$.
These spaces $W_k$ are preserved by $O(2)$, because this group is generated
by $SO(2)$ and $x \mapsto -x$. Thus, the decomposition of $L^2(S^1)$ into
irreducible representations of $O(2)$ is
\[
L^2(S^1) = \hatbigoplus_{k \ge 0} W_k.
\]
This decomposition is just slightly more complicated than the one for
$SO(2)$, because $\dim W_k = 2$ when $k>0$.

Another way to think of this equation is as the spectral decomposition of
the Laplacian operator $d^2/dx^2$.  Specifically,
\[
\frac{d^2}{dx^2} e^{ikx} = -k^2 e^{ikx}.
\]
Thus, $W_k$ is the eigenspace with eigenvalue $-k^2$.  The Laplacian plays a
fundamental role, since it is invariant under the action of $O(2)$.  (In
other words, translating or reflecting a function commutes with taking its
Laplacian.) In fact, the Laplacian generates the algebra of
isometry-invariant differential operators on $S^1$, but that's going
somewhat far afield from anything we will need.

\section{Fourier series on a torus} \label{sec:fouriertorus}

The $S^1$ theory generalizes pretty straightforwardly if we think of $S^1$
as a one-dimensional torus.  We can view a higher-dimensional flat torus as
$\R^n/\Lambda$, where $\Lambda$ is a lattice in $\R^n$.  In other words, we
simply take a fundamental cell for $\Lambda$ and wrap around whenever we
cross the boundary.  When we looked at $S^1$, we wrote it as $\R^1/\Lambda$
with $\Lambda = 2\pi\Z$, and it's worth keeping this example in mind.

We can decompose $L^2(\R^n/\Lambda)$ into exponential functions in exactly
the same way as we did for $S^1$. It works out particularly simply since
$\R^n/\Lambda$ is an abelian group. To write this decomposition down, we
need to figure out which exponential functions are periodic modulo
$\Lambda$.  Suppose $y \in \R^n$, and consider the exponential function
\[
x \mapsto e^{2\pi i \langle x,y \rangle}
\]
from $\R^n$ to $\C$.  Here $\langle \cdot, \cdot \rangle$ denotes the usual
inner product on $\R^n$ (not the inner product on functions used in the
previous section).  This formula defines a function on $\R^n/\Lambda$ if and
only if it is invariant under translation by vectors in $\Lambda$.

What happens if we translate the function $x \mapsto e^{2\pi i \langle x,y
\rangle}$ by a vector $z$?  It gets multiplied by $e^{-2\pi i \langle z,y
\rangle}$, and so it is always an eigenfunction of the translation operator.
Furthermore, it is invariant under translation by vectors in $\Lambda$ if
and only if $y$ satisfies
\[
e^{2\pi i \langle z,y \rangle} = 1
\]
for all $z \in \Lambda$, which is equivalent to $\langle z,y \rangle \in \Z$
for all $z \in \Lambda$.

Let
\[
\Lambda^* = \{y \in \R^n \ |\  \langle z,y \rangle \in \Z \textup{ for all $z \in \Lambda$}\}
\]
be the \emph{dual lattice} to $\Lambda$.  Thus, the exponential functions
that are periodic modulo $\Lambda$ are parameterized by $\Lambda^*$.

\begin{exercise} \label{ex:duallattice}
Prove that $\Lambda^*$ is a lattice.  Specifically, prove that if
$v_1,\dots,v_n$ is any basis of $\Lambda$, then $\Lambda^*$ has
$v_1^*,\dots,v_n^*$ as a basis, where these vectors are the dual basis
vectors satisfying
\[
\langle v_i, v_j^* \rangle = \begin{cases} 1 & \textup{if $i=j$, and}\\
0 & \textup{otherwise.}
\end{cases}
\]
Deduce also that $(\Lambda^*)^* = \Lambda$.
\end{exercise}

Let $V_y$ be the complex multiples of $x \mapsto e^{2\pi i \langle x,y
\rangle}$. Then
\[
L^2(\R^n/\Lambda) = \hatbigoplus_{y \in \Lambda^*} V_y,
\]
which is the decomposition into irreducible representations under the
translation action.

When $n=1$, the lattice $\Lambda$ is determined up to scaling.  In the
previous section we took $\Lambda = 2 \pi \Z$, in which case $\Lambda^* =
(2\pi)^{-1} \Z$.  The elements of $\Lambda^*$ are $(2\pi)^{-1}k$, where $k$
is an integer, and $V_{(2\pi)^{-1}k}$ is spanned by $x \mapsto e^{ikx}$.
Thus, we recover exactly the same theory as in the previous section, except
that we now write $V_{(2\pi)^{-1}k}$ instead of $V_k$.  It's arguably
prettier to take $\Lambda = \Z$ and use the functions $x \mapsto e^{2 \pi i
k x}$, but this is a matter of taste.

The higher-dimensional analogue of the $O(2)$ theory is a little more subtle.
The map $x \mapsto -x$ is always a symmetry of $\R^n/\Lambda$, and taking it
into account means combining $V_y$ with $V_{-y}$ as before.  Generically, all
the symmetries of $\R^n/\Lambda$ are generated by translations and $x \mapsto
-x$.  However, particularly nice lattices may have further symmetries.  If
$G$ is the automorphism group of the lattice itself, then the full group of
isometries of $\R^n/\Lambda$ is the semidirect product of $G$ with the
additive group $\R^n/\Lambda$.  What effect this has on the decomposition of
$L^2(\R^n/\Lambda)$ depends on the representation theory of $G$.  However,
for many purposes this is not important, and the decomposition under
translations alone will suffice.

\section{Spherical harmonics}

If we think of $S^1$ as a one-dimensional sphere, rather than a
one-dimensional torus, then it is less clear how to generalize Fourier
series.  Instead of exponential functions, we'll have to use \emph{spherical
harmonics}.

The symmetry group of the unit sphere
\[
S^{n-1} = \{x \in \R^n \ |\  |x|^2 = 1\}
\]
is the orthogonal group $O(n)$, which consists of $n \times n$ orthogonal
matrices.  As before, $L^2(S^{n-1})$ is a Hilbert space under the inner
product
\[
\langle f,g \rangle = \int_{S^{n-1}} \overline{f(x)} g(x) \, dx,
\]
where the integral is taken with respect to the surface measure on
$S^{n-1}$, and $L^2(S^{n-1})$ is a unitary representation of $O(n)$. We
would like to decompose it into irreducible representations of $O(n)$.

To get a handle on $L^2(S^{n-1})$, we will study the polynomials on
$S^{n-1}$.  Let $\mathcal{P}_k$ be the subset of $L^2(S^{n-1})$ consisting
of polynomials on $\R^n$ of total degree at most $k$.  (Strictly speaking,
it consists of the restrictions of these polynomials to $S^{n-1}$, since two
different polynomials can define the same function on the unit sphere.) Then
\[
\mathcal{P}_0 \subseteq \mathcal{P}_1 \subseteq \mathcal{P}_2 \subseteq \dots,
\]
and each $\mathcal{P}_k$ is a representation of $O(n)$.  To see why, note
that rotating or reflecting a polynomial gives another polynomial of the same
degree; in fact, this is true for any invertible linear transformation.

Let $W_0 = \mathcal{P}_0$, and for $k>0$ let $W_k$ be the orthogonal
complement of $\mathcal{P}_{k-1}$ in $\mathcal{P}_k$.  Then $W_k$ is a
representation of $O(n)$, because $\mathcal{P}_{k-1}$ and $\mathcal{P}_k$
are representations and the inner product in $L^2(S^{n-1})$ is
$O(n)$-invariant.  Iterating this decomposition shows that
\[
\mathcal{P}_k = W_0 \oplus W_1 \oplus \dots \oplus W_k.
\]
Furthermore, $\bigcup_k \mathcal{P}_k$ is dense\footnote{Continuous
functions are dense in $L^2(S^{n-1})$, and the Stone-Weierstrass theorem
tells us that polynomials are dense in the space of continuous functions.}
in $L^2(S^{n-1})$, and hence
\[
L^2(S^{n-1}) = \hatbigoplus_{k \ge 0} W_k.
\]
We have thus decomposed $L^2(S^{n-1})$ into finite-dimensional
representations of $O(n)$.  In fact they are irreducible, as we will see in
the next lecture, but that fact is by no means obvious.

This decomposition may sound abstract, but it's actually quite elementary,
since it is simply given by polynomials.  Let's check that it agrees with
what we did for $S^1$. Polynomials on $\R^2$ can be written in terms of the
coordinate variables $x$ and $y$, and in the $\R/2\pi\Z$ picture we have $x
= \cos \theta$ and $y = \sin \theta$ with $\theta \in \R/2\pi\Z$.  Thus,
$\mathcal{P}_k$ consists of polynomials of degree at most $k$ in the
functions $\theta \mapsto \cos \theta$ and $\theta \mapsto \sin \theta$.  If
we write $\cos \theta = (e^{i\theta} + e^{-i\theta})/2$ and $\sin \theta =
(e^{i\theta} - e^{-i\theta})/(2i)$, then we find that the elements of
$\mathcal{P}_k$ involve powers of $e^{i\theta}$ ranging from $-k$ to $k$,
and every such power is in $\mathcal{P}_k$. In other words,
\[
\mathcal{P}_k = V_{-k} \oplus V_{-(k-1)} \oplus \dots \oplus V_{k-1} \oplus V_{k}
\]
in the notation from \S\ref{sec:fourierseries}.  In particular, the
orthogonal complement $W_k$ of $\mathcal{P}_{k-1}$ in $\mathcal{P}_k$ is
indeed $V_{-k} \oplus V_k$ when $k>0$, which agrees with our previous
construction.

Returning to $L^2(S^{n-1})$, we call the elements of $W_k$ \emph{spherical
harmonics} of degree $k$.  Note that the word ``harmonic'' generalizes the
term from music theory for a note whose frequency is an integer multiple of
the base frequency; this term literally describes $W_k$ when $n=2$, and it is
applied by analogy in higher dimensions.

Writing $W_k$ down explicitly is a little subtle, because two different
polynomials on $\R^n$ can restrict to the same function on $S^{n-1}$.  For
example, $x_1^2 + \dots + x_n^2$ and $1$ are indistinguishable on the unit
sphere.  To resolve this ambiguity, we will choose a canonical representative
for each equivalence class.

\begin{lemma}
For each polynomial on $\R^n$, there is a unique harmonic polynomial on
$\R^n$ with the same restriction to $S^{n-1}$.
\end{lemma}

Recall that \emph{harmonic} means $\Delta g = 0$, where
\[
\Delta = \frac{\partial^2}{\partial x_1^2} + \dots + \frac{\partial^2}{\partial x_n^2}
\]
is the Laplacian on $\R^n$.  If $f$ is a harmonic polynomial with
$f|_{S^{n-1}} = g|_{S^{n-1}}$, then $f$ is called a \emph{harmonic
representative} for $g$.

The main fact we'll need about harmonic functions is the maximum principle
\cite[p.~7]{ABR2001}: the maximum of a harmonic function on a domain $D$
cannot occur in the interior of $D$ (instead, it must occur on the
boundary).  Of course, multiplying the function by $-1$ shows that the same
is true for the minimum.

\begin{proof}
Uniqueness follows immediately from the maximum principle: if
\[
g_1|_{S^{n-1}} = g_2|_{S^{n-1}}
\]
with both $g_1$ and $g_2$ harmonic, then
$g_1-g_2$ is a harmonic function that vanishes on $S^{n-1}$. It must
therefore vanish inside the sphere as well (since its minimum and maximum
over the ball must be attained on the sphere), which implies that $g_1=g_2$
because they are polynomials.

Proving existence of a harmonic representative is only slightly trickier. Let
$\mathcal{Q}_k$ denote the space of polynomials of degree at most $k$ on
$\R^n$.  Note that the difference between $\mathcal{P}_k$ and $\mathcal{Q}_k$
is that $\mathcal{P}_k$ consists of the restrictions to $S^{n-1}$, and thus
$\mathcal{P}_k$ is the quotient of $\mathcal{Q}_k$ by the polynomials whose
restrictions vanish. Multiplication by $x_1^2+\dots+x_n^2-1$ maps
$\mathcal{Q}_{k-2}$ injectively to $\mathcal{Q}_k$, and its image vanishes on
$S^{n-1}$, so
\[
\dim \mathcal{P}_k \le \dim \mathcal{Q}_k - \dim \mathcal{Q}_{k-2}.
\]
On the other hand, $\Delta$ maps $\mathcal{Q}_k$ to $\mathcal{Q}_{k-2}$, and
hence
\[
\dim \ker \Delta|_{\mathcal{Q}_k} \ge \dim \mathcal{Q}_k - \dim \mathcal{Q}_{k-2} \ge \dim \mathcal{P}_k.
\]
By uniqueness, the restriction map from $\ker \Delta|_{\mathcal{Q}_k}$ to
$\mathcal{P}_k$ is injective, and thus the inequality $\dim \ker
\Delta|_{\mathcal{Q}_k} \ge \dim \mathcal{P}_k$ implies that each polynomial
in $\mathcal{P}_k$ must have a harmonic representative (and $\dim \ker
\Delta|_{\mathcal{Q}_k} = \dim \mathcal{P}_k$).
\end{proof}

Another way to understand spherical harmonics is as eigenfunctions of the
\emph{spherical Laplacian} $\Delta_{S^{n-1}}$, which acts on $C^2$ functions
on the sphere (i.e., twice continuously differentiable functions).  The right
setting for this operator is the theory of Laplace-Beltrami operators in
Riemannian geometry, but we can give a quick, ad hoc definition as follows.
Given a function $f$ on $S^{n-1}$, extend it to a radially constant function
$f_{\mbox{\scriptsize radial}}$ on $\R^n \setminus \{0\}$. Then we define
$\Delta_{S^{n-1}}$ by
\[
\Delta_{S^{n-1}} f = \big(\Delta
f_{\mbox{\scriptsize radial}}\big)\big|_{S^{n-1}}.
\]
In other words, $\Delta_{S^{n-1}} f$ measures the Laplacian of $f$ when there
is no radial change.

It is often notationally convenient to extend the operator $\Delta_{S^{n-1}}$
to apply to functions $f \colon \R^n\setminus\{0\} \to \R$, rather than just
functions defined on the unit sphere. We can do so by rescaling everything to
the unit sphere. More precisely, to define $\Delta_{S^{n-1}} f$ at the point
$x$, we consider the function $g \colon S^{n-1} \to \R$ defined by $g(y) =
f(|x|y)$, and we let
\[
\Delta_{S^{n-1}} f(x)  =  \Delta_{S^{n-1}} g(x/|x|).
\]

The advantage of being able to apply $\Delta_{S^{n-1}}$ to functions on
$\R^n \setminus\{0\}$ is that it becomes the angular part of the Euclidean
Laplacian in spherical coordinates:

\begin{exercise}
Prove that if $r$ denotes the distance to the origin and $\partial/\partial
r$ is the radial derivative, then for every $C^2$ function $f \colon \R^n
\to \R$,
\begin{equation} \label{eq:sphericalcoors}
\Delta f = \frac{\partial^2 f}{\partial r^2} + \frac{n-1}{r} \frac{\partial f}{\partial r} + \frac{1}{r^2} \Delta_{S^{n-1}} f
\end{equation}
when $r \ne 0$.
\end{exercise}

If $f$ is homogeneous of degree $k$, then \eqref{eq:sphericalcoors} becomes
\[
\Delta f = \frac{k(k-1)f}{r^2} + \frac{(n-1)kf}{r^2} + \frac{1}{r^2} \Delta_{S^{n-1}} f.
\]
Then $\Delta f = 0$ is equivalent to $\Delta_{S^{n-1}} f = -k(k+n-2) f$. In
other words, harmonic functions that are homogeneous of degree $k$ are
eigenfunctions of the spherical Laplacian with eigenvalue $-k(k+n-2)$.  We
will see shortly that the spherical harmonics in $W^k$ are all homogeneous of
degree $k$, and thus that the spaces $W_k$ are the eigenspaces of
$\Delta_{S^{n-1}}$.

First note, that the Euclidean Laplacian maps homogeneous polynomials of
degree $k$ to homogeneous polynomials of degree $k-2$.  Thus, every harmonic
polynomial is the sum of homogeneous harmonics.

In terms of spherical harmonics, $\mathcal{P}_k$ is the sum of the
eigenspaces of $\Delta_{S^{n-1}}$ with eigenvalues $-\ell(\ell+n-2)$ for
$\ell=0,1,\dots,k$.  These eigenspaces are orthogonal, because the spherical
Laplacian is symmetric:

\begin{lemma}
For $C^2$ functions $f$ and $g$ on $S^{n-1}$,
\[
\langle  f, \Delta_{S^{n-1}}g \rangle = \langle \Delta_{S^{n-1}} f,g \rangle.
\]
\end{lemma}

\begin{proof}
This identity is well-known for the Laplace-Beltrami operator, but verifying
it using our ad hoc definition takes a short calculation.  Replace $f$ and
$g$ with their radial extensions to $\R^n \setminus\{0\}$, and let
\[
\Omega = \{x \in \R^n \ |\  1/2 \le |x| \le 2\}.
\]
Equation \eqref{eq:sphericalcoors} implies that
\[
\big(\langle  f, \Delta_{S^{n-1}}g \rangle - \langle \Delta_{S^{n-1}} f,g \rangle\big) \int_{1/2}^2 \omega_n r^{n-3} \, dr
= \int_\Omega f \Delta g - g \Delta f,
\]
where the integral over $\Omega$ is with respect to Lebesgue measure and
$\omega_n$ is the surface area of $S^{n-1}$.  In this equation, $\omega_n
r^{n-3}$ combines the volume factor from spherical coordinates with the
$1/r^2$ factor multiplying $\Delta_{S^{n-1}} f$ in \eqref{eq:sphericalcoors}.

Now Green's identity tells us that
\[
\int_{\Omega} f \Delta g - g \Delta f = \int_{\partial \Omega} f \frac{\partial g}{\partial n} - g \frac{\partial f}{\partial n},
\]
where $\partial/\partial n$ denotes the normal derivative and the integral
over $\partial \Omega$ is with respect to surface measure. It vanishes
because $\partial f/\partial n =
\partial g/\partial n = 0$ by construction.
\end{proof}

Because $\mathcal{P}_k$ is the sum of the eigenspaces of $\Delta_{S^{n-1}}$
with eigenvalues $-\ell(\ell+n-2)$ for $\ell=0,1,\dots,k$ and these
eigenspaces are orthogonal, the orthogonal complement of $\mathcal{P}_{k-1}$
in $\mathcal{P}_k$ must be the $-k(k+n-2)$ eigenspace.  Thus, $W_k$ consists
of the harmonic polynomials that are homogeneous of degree $k$.  This gives
a rather concrete, if cumbersome, description of the space of spherical
harmonics. By contrast, people sometimes make spherical harmonics look
unnecessarily exotic by writing them in spherical coordinates as
eigenfunctions of the Laplacian.

\begin{exercise}
Compute the homogeneous harmonic polynomials explicitly when $n=2$, and
check that this computation agrees with our earlier analysis of $S^1$.
\end{exercise}

We will see in the next lecture that $W_k$ is an irreducible representation
of $O(n)$. Thus, we have found the complete decomposition of $L^2(S^{n-1})$
into irreducible representations, as well as the spectral decomposition of
the Laplacian. The biggest conceptual difference from $S^1$ is that the space
$W_k$ of degree $k$ spherical harmonics has much higher dimension than $2$ in
general, but that's not an obstacle to using this theory.

\lecture{Energy and packing bounds on spheres}

\section{Introduction}

In this lecture, we will use spherical harmonics to prove bounds for packing
and energy minimization on spheres.\footnote{Analogous techniques work in
various other settings, such as projective spaces or Grassmannians.} Our
technique will be essentially the same as in the proof of
Proposition~\ref{prop:simplex} from the second lecture, but the bounds will
be more sophisticated algebraically and much more powerful. By the end of the
lecture we will be able to solve the kissing problem in $\R^8$ and $\R^{24}$,
as well as analyze almost all of the known cases of universal optimality. In
the next lecture we will tackle Euclidean space using much the same approach,
but the analytic technicalities will be greater and it will be useful to have
looked at the spherical case first.

Everything we do will be based on studying the distances that occur between
pairs of points.  Motivated by error-correcting codes, we call a finite
subset $\mC$ of $S^{n-1}$ a \emph{code}. The \emph{distance distribution} of
a code measures how often each pairwise distance occurs.  For $-1 \le t \le
1$, define the distance distribution $A$ of $\mC$ by
\[
A_t = \#\big\{ (x,y) \in \mC^2 \ |\  \langle x,y \rangle = t\big\},
\]
where $\langle \cdot,\cdot\rangle$ denotes the usual inner product on $\R^n$.
Recall that $|x-y|^2 = 2-\langle x,y \rangle$ when $x$ and $y$ are unit
vectors; thus, $A_t$ counts the number of pairs at distance $\sqrt{2-2t}$,
but inner products are a more convenient way to index these distances.  In
physics terms \cite[p.~63]{Torquato2002}, the distance distribution is
equivalent to the \emph{pair correlation function}, although it is formulated
a little differently.

We can express the energy for a pair potential function $f$ in terms of the
distance distribution via
\begin{equation} \label{eq:EnergyAt}
\sum\limits_{\substack{x,y
\in \mC\\ x \ne y}} f(|x-y|^2) = \sum_{-1 \le t < 1} f(2-2t) A_t.
\end{equation}
(In the sum on the right, there are uncountably many values of $t$, but only
finitely many of the summands are nonzero.  Note that the restriction to
$t<1$ is to avoid self-interactions; it corresponds to $x \ne y$ on the
left.) Thus, figuring out which energies can be attained amounts to
understanding what the possible pair correlation functions are. Which
constraints must they satisfy?

We have made an important trade-off here.  The dependence of energy on the
distance distribution is as simple as possible, because the right side of
\eqref{eq:EnergyAt} is a linear function of the variables $A_t$.  However,
the nonlinearity in this problem cannot simply disappear.  Instead, it
reappears in the question of which distance distributions occur for actual
point configurations.

There are some obvious constraints for an $N$-point code: $A_t \ge 0$ for
all $t$, $A_1=N$, and $\sum_t A_t = N^2$.  They follow trivially from the
definition
\[
A_t = \#\big\{ (x,y) \in \mC^2 \ |\  \langle x,y \rangle = t\big\}.
\]
Another obvious constraint is that $A_t$ must be an integer for each $t$, but
we will generally ignore this constraint, because optimization theory does
not handle integrality constraints as seamlessly as it handles inequalities.

There are also less obvious constraints, such as
\[
\sum_t A_t t \ge 0.
\]
To see why this inequality holds, note that
\[
\sum_t A_t t = \sum_{x,y \in \mC} \langle x,y \rangle,
\]
because $A_t$ counts how often $t$ occurs as an inner product between points
in $\mC$.  Thus,
\[
\sum_t A_t t = \sum_{x,y \in \mC} \langle x,y \rangle =
\left\langle\sum_{x \in C} x, \sum_{y \in C} y \right\rangle =
\left|\sum_{x \in C} x\right|^2 \ge 0.
\]
Recall that this is the inequality we used to analyze simplices at the end of
the second lecture.

Delsarte discovered an infinite sequence of linear inequalities generalizing
this one.\footnote{Delsarte's initial discovery was in a discrete setting
\cite{Delsarte1972}, but analogous techniques apply to spheres
\cite{DGS1977,KabatianskyLevenshtein1978}.} The factor of $t$ above is
replaced with certain special functions, namely \emph{Gegenbauer} or
\emph{ultraspherical polynomials}, which are a family $P^n_k$ of polynomials
in one variable with $\deg(P^n_k) = k$. The Delsarte inequalities then say
that whenever $A$ is the distance distribution of a configuration in
$S^{n-1}$,
\begin{equation} \label{eq:Geg1}
\sum_t A_t P^n_k(t) \ge 0
\end{equation}
for all $k$. In particular, $P^n_1(t)=t$, from which we recover the previous
inequality, and $P^n_0(t) = 1$, while the higher-degree polynomials depend
on $n$.

The Delsarte inequalities are far from a complete characterization of the
distance distributions of codes. However, they are particularly beautiful
and important constraints on these distance distributions.

An equivalent reformulation of \eqref{eq:Geg1} is that for every finite set
$\mC \subset S^{n-1}$,
\[
\sum_{x,y \in \mC} P^n_k(\langle x,y \rangle) \ge 0.
\]
We will return in \S\ref{sec:ultraspherical} to what ultraspherical
polynomials are and why they have this property.  In the meantime, we will
treat them as a black box while we explore how the Delsarte inequalities are
used to prove bounds.

\section{Linear programming bounds}

The energy of a code is given by the linear function
\[
\frac{1}{2} \sum_{-1 \le t < 1} f(2-2t) A_t
\]
of its distance distribution, and the Delsarte inequalities
\[
\sum_t A_t P^n_k(t) \ge 0
\]
are linear in $A$ as well.  The \emph{linear programming bounds} minimize
the energy subject to these linear constraints.\footnote{Recall that
``linear programming'' means optimizing a linear function subject to linear
constraints. There are efficient algorithms to solve finite linear
programs.} Because of the linearity, these bounds are particularly well
behaved and useful. The only computational difficulty is that there are
infinitely many variables $A_t$.

Let's write down the linear programming bounds more precisely. To begin
with, we are given the dimension $n$, number $N$ of points, and potential
function $f$.  Then linear programming bounds attempt to choose $A_t$ for
$-1 \le t \le 1$ so as to minimize
\[
\frac{1}{2} \sum_{-1 \le t < 1} A_t f(2-2t)
\]
subject to
\begin{align*}
A_1 &= N,\\
A_t &\ge 0 \text{ for $-1 \le t \le 1$},\\
\sum_t A_t &= N^2, \text{ and}\\
\sum_t A_t P^n_k(t) &\ge 0 \text{ for all $k \ge 1$.}
\end{align*}
This optimization problem gives us a lower bound for the energy of codes in
$S^{n-1}$, because every code has a corresponding distance distribution.
However, there is no reason to expect the bound to be sharp in general: the
optimal choice of $A_t$ will usually not even be integral, let alone come
from an actual code. Of course one could improve the bound by imposing
integrality, but then the optimization problem would become far less
tractable.  In particular, it would no longer be a convex optimization
problem.

Linear programming bounds are well suited to computer calculations, but they
have not yet been fully optimized.  Any given case can be solved numerically,
but the general pattern is unclear.  In particular, for most $n$, $N$, and
$f$ we do not know the optimal solution.

In practice, it is useful to apply linear programming duality, in which we
try to prove bounds on energy by taking linear combinations of the
constraints.  If we multiply the Delsarte inequalities
\[
\sum_t A_t P^n_k(t) \ge 0
\]
by constants $h_k$ and then sum over $k$, we obtain the following theorem.

\begin{theorem}[Yudin \cite{Yudin1993}] \label{thm:yudin}
Suppose $h = \sum_k h_k P^n_k$ with $h_k \ge 0$ for $k \ge 1$, and suppose
$h(t) \le f(2-2t)$ for $t \in [-1,1)$. Then every $N$-point configuration
$\mC$ on $S^{n-1}$ satisfies
\[
\sum\limits_{\substack{x,y
\in \mC\\ x \ne y}} f(|x-y|^2) \ge N^2 h_0 - N h(1).
\]
\end{theorem}

The auxiliary function $h$ is generally a polynomial, in which case $h_k=0$
for all sufficiently large $k$, but convergence of $\sum_k h_k P^n_k$ on
$[-1,1]$ suffices.  (It turns out that $|P^n_k| \le P^n_k(1)$ on $[-1,1]$,
and hence the convergence is automatically absolute and uniform.)

\begin{proof}
We have
\begin{align*}
\sum\limits_{\substack{x,y
\in \mC\\ x \ne y}} f(|x-y|^2) &\ge \sum\limits_{\substack{x,y
\in \mC\\ x \ne y}} h(\langle x,y \rangle) \qquad \text{(because $f(2-2t) \ge h(t)$ pointwise)}\\
&= \sum_{x,y \in \mC} h(\langle x,y \rangle) - Nh(1)\\
&= N^2h_0 - Nh(1) + \sum_{k\ge 1} h_k \sum_{x,y \in \mC} P^n_k(\langle x,y \rangle)\\
&\ge N^2h_0 - Nh(1),
\end{align*}
as desired.
\end{proof}

Note that the proof rests on the fundamental inequality
\[
\sum_{x,y \in \mC} P^n_k(\langle x,y \rangle) \ge 0.
\]
The proof technique might seem extraordinarily wasteful, since it involves
throwing away many terms in our sum. However, $P^n_k(\langle x,y \rangle)$
averages to zero over the whole sphere when $k \ge 1$, which suggests that
the double sums
\[
\sum_{x,y \in \mC} P^n_k(\langle x,y \rangle)
\]
may not be so large after all when $\mC$ is well distributed over the
sphere.

Theorem~\ref{thm:yudin} tells us that to prove a lower bound for $f$-energy,
all we need is a lower bound $h$ for the potential function $f$ such that $h$
has non-negative ultraspherical coefficients.  Such an auxiliary function is
a convenient certificate for a lower bound.

Outside of a few special cases, nobody knows the optimal $h$ for a given $f$.
However, numerical optimization is an effective way to compute approximations
to it. One can use more sophisticated techniques such as sums of squares and
semidefinite programming, but even the most straightforward approach works
well in practice: let $h$ be a polynomial of degree $d$, and instead of
imposing the inequality $h(t) \le f(2-2t)$ for all $t$, impose it just at
finitely many locations (chosen fairly densely in $[-1,1)$, of course). Then
we are left with a finite linear program, i.e., a linear optimization problem
with only finitely many variables and constraints, which is easily solved
numerically using standard software. The resulting auxiliary function $h$
might not satisfy $h(t) \le f(2-2t)$ everywhere, but any violations will be
small, and we can eliminate them by adjusting the constant term $h_0$ without
substantially changing the energy bound.

\section{Applying linear programming bounds}

Linear programming bounds are behind almost every case in which universal
optimality, or indeed any sharp bound on energy, is known.  As mentioned
above, they are generally far from sharp, but for certain codes they
miraculously give sharp bounds.  This is the case for all the universal
optima listed in Table~\ref{tab:univopt} from the second lecture.

When could the bound be sharp for a configuration $\mC$?  Equality holds in
Theorem~\ref{thm:yudin} iff every term we throw away in the proof is
actually already zero.  Inspecting the proof leads to the following
criteria:

\begin{lemma} \label{lemma:sharp}
The energy lower bound in Theorem~\ref{thm:yudin} is attained by a code
$\mC$ if and only if $f(|x-y|^2) = h(\langle x,y \rangle)$ for all $x,y \in
\mC$ with $x \ne y$, and
\[
\sum_{x,y \in \mC} P^n_k(\langle x,y \rangle) = 0
\]
for all $k \ge 1$ for which $h_k > 0$.
\end{lemma}

The first condition says that $h(t) = f(2-2t)$ whenever $t = \langle x,y
\rangle$ with $x,y \in \mC$ and $x \ne y$.  Because $h(t) \le f(2-2t)$ for
all $t$, the functions $h$ and $f$ cannot cross.  Instead, they must agree to
order at least $2$ whenever they touch.

In practice, sharp bounds are usually obtained in the simplest possible way
based on this tangency constraint. We choose $h$ to be a polynomial of as low
a degree as possible subject to agreeing with $f$ to order $2$ at each inner
product that occurs between distinct points in $\mC$.  This specifies a
choice of $h$, but it is not obvious that it has any of the desired
properties. For example, the inequality $h(t) \le f(2-2t)$ might be violated
in between the points at which we force equality, and there is no obvious
reason to expect the ultraspherical coefficients $h_k$ to be nonnegative.

This construction of $h$ is generally far from optimal when it works at all,
but for particularly beautiful codes it does remarkably well at proving sharp
bounds.  For example, let's show that regular simplices are universally
optimal, which was Exercise~\ref{ex:simplicesunivopt} from the second
lecture. Recall that for $N \le n+1$, the $N$-point regular simplex $\mC$ in
$S^{n-1}$ has all inner products equal to $-1/(N-1)$.

\begin{proposition}
For $N \le n+1$, the $N$-point regular simplex is universally optimal in
$S^{n-1}$.
\end{proposition}

We'll describe the proof in terms of linear programming bounds, but one could
reword it to use just the inequality
\[
\left|\sum_{x \in \mC} x \right|^2 \ge 0
\]
(as was intended in Exercise~\ref{ex:simplicesunivopt} from the second
lecture).

\begin{proof}
We will show that the simplex in fact minimizes energy for every decreasing,
convex potential function $f$, which is an even stronger property than
universal optimality.  Let $h(t)$ be the tangent line to $f(2-2t)$ at
$t=-1/(N-1)$; in other words,
\[
h(t) = f\big(2+2/(N-1)\big) - 2 f'\big(2+2/(N-1)\big)\big(t + 1/(N-1)\big).
\]
This function is the lowest-degree polynomial that agrees with $f(2-2t)$ to
order $2$ at all the inner products occurring in the regular simplex, which
makes it a special case of the construction outlined above.

Because $f$ is convex, $h(t) \le f(2-2t)$ for all $t$.  Thus, the first
inequality we need for $h$ does in fact hold.  To check the nonnegativity of
the ultraspherical coefficients (aside from the constant term), note that
the first two ultraspherical polynomials are $1$ and $t$.  If we express
$h(t)$ in terms of this basis, then the coefficient of $t$ is $- 2
f'\big(2+2/(N-1)\big)$, which is nonnegative since $f$ is decreasing.  Thus,
$h$ satisfies the hypotheses of Theorem~\ref{thm:yudin}. Furthermore,
$h(t)=f(2-2t)$ when $t = -1/(N-1)$ by construction, and
\[
\sum_{x,y \in \mC} \langle x,y \rangle = \left|\sum_{x \in \mC} x \right|^2 = 0.
\]
These are the conditions for a sharp bound in Lemma~\ref{lemma:sharp}, and
so we conclude that our energy bound is equal to the energy of the regular
simplex. Hence regular simplices minimize energy for all decreasing, convex
potential functions, and in particular they are universally optimal.
\end{proof}

Codes with more inner products are more complicated to handle, but in any
given case one can figure out whether this approach works.  If one analyzes
the technique in sufficient generality, it proves the following theorem,
which extends a theorem of Levenshtein \cite{Levenshtein1992}.

\begin{theorem}[Cohn and Kumar \cite{CohnKumar2007}]
\label{thm:cohnkumar2007}
Every $m$-distance set that is a spherical $(2m-1)$-design is universally
optimal.
\end{theorem}

Here an \emph{$m$-distance set} is a set in which $m$ distances occur
between distinct points, and a \emph{spherical $k$-design} is a finite
subset $\mathcal{D}$ of the sphere $S^{n-1}$ such that for every polynomial
$p \colon \R^n \to \R$ of total degree at most $k$, the average of $p$ over
$\mathcal{D}$ is equal to its average over the entire sphere $S^{n-1}$. In
other words, averaging at the points of $\mathcal{D}$ is an exact numerical
integration formula for polynomials up to degree $k$, which means these
points are exceedingly well distributed over the sphere.

This theorem suffices to handle every known universal optimum on the surface
of a sphere (see Table~\ref{tab:univopt} in the second lecture) except the
regular $600$-cell, which is dealt with in \S7 of \cite{CohnKumar2007}.
Surely that's not the only exception, but it is unclear where to find other
universal optima that go beyond Theorem~\ref{thm:cohnkumar2007}.

\section{Spherical codes and the kissing problem}

Recall that the spherical code problem asks whether $N$ points can be
arranged on $S^{n-1}$ so that no two are closer than angle $\theta$ to each
other along the great circle connecting them.  In other words, the
\emph{minimal angle} between the points is at least $\theta$.  This is a
packing problem: how many spherical caps of angular radius $\theta/2$ can we
pack on the surface of a sphere?

The most famous special case is the kissing problem discussed in the first
lecture.  Given a central unit ball, the kissing problem asks how many
non-overlapping unit balls can be arranged tangent to it.  Equivalently, the
points of tangency should form a spherical code with minimal angle at least
$60^\circ$ (see Figure~\ref{fig:kiss}).

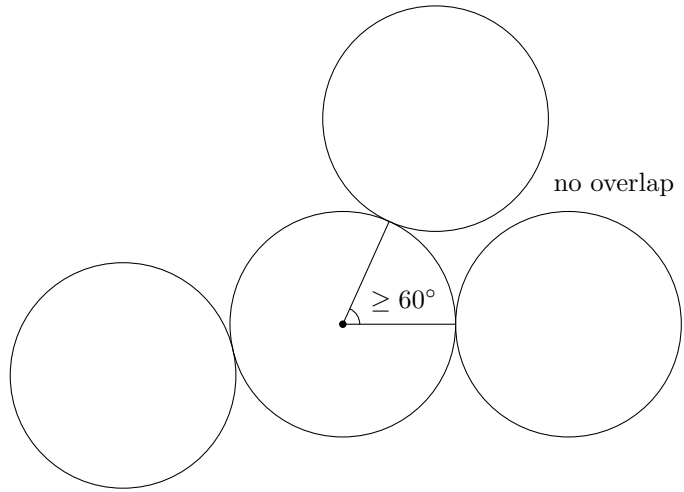
\begin{figure}
\begin{center}
\begin{tikzpicture}[scale=1.5]
  \draw (0,0) circle (1);
  \draw (2,0) circle (1);
  \draw (0.822087,1.823231) circle (1);
  \draw (-1.947645,-0.454404) circle (1);
  \draw (0,0)--(1,0);
  \draw (0,0)--(0.411043,0.911615);
  \fill (0,0) circle (0.03333);
  \draw (0.15,0) arc (0:65.729:0.15) ;
  \draw (0.54,0.21) node (a) {${\ge 60^\circ}$};
  \draw (2.4,1.24) node (b) {no overlap};
\end{tikzpicture}
\end{center}
\caption{An angle of $60^\circ$ or more between tangent spheres is equivalent
to avoiding overlap between them.} \label{fig:kiss}
\end{figure}

Linear programming bounds apply to this problem.  In fact, packing problems
were the original application for these bounds \cite{Delsarte1972}, before
Yudin applied them to energy minimization \cite{Yudin1993}.

\begin{theorem} \label{thm:sphercode}
Suppose $h = \sum_k h_k P^n_k$ with $h_k \ge 0$ for $k \ge 0$ and $h_0>0$,
and suppose $h(t) \le 0$ for $t \in [-1,\cos \theta]$. Then every code $\mC$
in $S^{n-1}$ with minimal angle at least $\theta$ satisfies
\[
|\mC| \le h(1)/h_0.
\]
\end{theorem}

\begin{proof}
We have
\[
|\mC| h(1) \ge \sum_{x,y \in \mC} h(\langle x,y \rangle) = \sum_k h_k \sum_{x,y \in \mC} P^n_k(\langle x,y \rangle) \ge |\mC|^2 h_0. \qedhere
\]
\end{proof}

As in the case of energy minimization, this bound is generally not sharp, but
on rare occasions we are lucky enough to get a sharp bound.  The most famous
case is the kissing problem in $\R^8$ and $\R^{24}$, which was solved
independent by Levenshtein \cite{Levenshtein1979} and Odlyzko and Sloane
\cite{OdlyzkoSloane1979}.  In particular, the kissing number is $240$ in
$\R^8$ and $196560$ in $\R^{24}$, as achieved by the $E_8$ lattice and the
Leech lattice.  It is not so difficult to prove these upper bounds using
Theorem~\ref{thm:sphercode}.  In particular, we take
\[
h(t) = (t+1)(t+1/2)^2t^2(t-1/2)
\]
in the $\R^8$ case, and
\[
h(t) = (t+1)(t+1/2)^2(t+1/4)^2t^2(t-1/4)^2(t-1/2)
\]
in the $\R^{24}$ case. (The roots correspond to the inner products that
occur in the kissing configurations.)  Checking that these polynomials
satisfy the hypotheses of Theorem~\ref{thm:sphercode} and prove sharp bounds
is a finite calculation.  Of course presenting it this way makes the proof
look like a miracle, and explaining it conceptually requires a deeper
analysis \cite{Levenshtein1992}.

\section{Ultraspherical polynomials}
\label{sec:ultraspherical}

So far, we have treated ultraspherical polynomials as a black box and taken
the Delsarte inequalities on faith.  In this section, we will finally examine
where these polynomials come from and why the inequalities hold.

One simple (albeit unmotivated) description is that ultraspherical
polynomials for $S^{n-1}$ are \emph{orthogonal polynomials} with respect to
the measure $(1-t^2)^{(n-3)/2} \, dt$ on $[-1,1]$. In other words,
\[
\int_{-1}^1 P^n_k(t) P^n_\ell(t) (1-t^2)^{(n-3)/2} \, dt = 0
\]
for $k \ne \ell$. Equivalently, $P^n_k$ is orthogonal to all polynomials of
degree less than $k$ with respect to this measure, because all such
polynomials are linear combinations of $P^n_0,\dots,P^n_{k-1}$.  We'll see
shortly where the measure comes from and why this orthogonality characterizes
the ultraspherical polynomials, but first let's explore its consequences.

Orthogonality uniquely determines the ultraspherical polynomials up to
scaling (and the scaling is irrelevant for our purposes, as long as we take
$P^n_k(1) > 0$ so as not to flip the Delsarte inequality).  Specifically, we
just apply Gram-Schmidt orthogonalization to $1, t, t^2, \dots$, which gives
an algorithm to compute these polynomials explicitly.  It's not the most
efficient method, but it works.

Although orthogonality may sound like an arcane property of a sequence of
polynomials, it has many wonderful and surprising consequences. For example,
it implies that $P^n_k$ has $k$ distinct roots in $[-1,1]$. To see why,
suppose $P^n_k$ changed sign at only $m$ points $r_1,\dots,r_m$ in $[-1,1]$,
with $m<k$. Then the polynomial
\[
P^n_k(t) (t-r_1) \dots (t-r_m)
\]
would never change sign on $[-1,1]$, which would contradict
\[
\int_{-1}^1 P^n_k(t) (t-r_1) \dots (t-r_m) (1-t^2)^{(n-3)/2} \,dt = 0
\]
(which holds because $(t-r_1)\dots(t-r_m)$ has degree less than $k$).  Thus,
$m=k$ and $P^n_k$ has $k$ distinct roots in $[-1,1]$, which means it's a
highly oscillatory function.

Although ultraspherical polynomials can be characterized via orthogonality,
it's not really a satisfactory explanation of where they come from.  To
explain that, we will use spherical harmonics.  Recall that as a
representation of $O(n)$, we can decompose $L^2(S^{n-1})$ as
\[
L^2(S^{n-1}) = \hatbigoplus_{k \ge 0} W_k,
\]
where $W_k$ consists of degree $k$ spherical harmonics.

We can obtain ultraspherical polynomials by studying the \emph{evaluation
map}: let $x \in S^{n-1}$, and consider the linear map that takes $f \in W_k$
to $f(x)$.  By duality for finite-dimensional vector spaces, this map must be
the inner product with some unique element $w_{k,x}$ of $W_k$, called a
\emph{reproducing kernel}. That is,
\[
f(x) = \langle w_{k,x},f \rangle
\]
for all $f \in W_k$.  Note that here $\langle \cdot,\cdot\rangle$ denotes
the inner product on $L^2(S^{n-1})$.  We will use the same notation for both
this inner product and the standard inner product on $\R^n$; to distinguish
between them, pay attention to which vector spaces their arguments lie in.

The function $w_{k,x}$ on $S^{n-1}$ has considerable structure.  For example,
it is invariant under all symmetries of $S^{n-1}$ that fix $x$:

\begin{lemma}
If $T$ is an element of $O(n)$ such that $Tx=x$, then $T w_{k,x} = w_{k,x}$.
\end{lemma}

\begin{proof}
This lemma follows easily from the invariance of the inner product on $W_k$
under $O(n)$.  We have $\langle w_{k,x},f \rangle = \langle Tw_{k,x}, f
\rangle$ for all $f \in W_k$, because
\[
\langle w_{k,x},f \rangle = f(x) = f(Tx) = (T^{-1}f)(x) = \langle w_{k,x},T^{-1} f \rangle =  \langle Tw_{k,x}, f \rangle,
\]
and hence $w_{k,x} = Tw_{k,x}$.
\end{proof}

Equivalently, $w_{k,x}(y)$ can depend only on the distance between $x$ and
$y$, and therefore it must be a function of $\langle x,y \rangle$ alone.  We
define $P^n_k$ by
\[
w_{k,x}(y) = P^n_k(\langle x,y \rangle).
\]
The reproducing kernel $w_{k,x}$ is a polynomial of degree $k$ in several
variables, because it is a spherical harmonic in $W_k$, and thus $P^n_k$
must be a polynomial of degree $k$ in one variable.  (Technically this
definition is off by a constant factor from the special case $P^n_1(t)=t$
mentioned earlier, but we could easily rectify that by rescaling so that
$P^n_k(1)=1$.)

We have finally explained where ultraspherical polynomials come from.  They
describe reproducing kernels for the spaces $W_k$, and the importance of
reproducing kernels is that they tell how to evaluate spherical harmonics at
points.

The drawback of the reproducing kernel definition is that it does not make
it clear how to compute these polynomials in any reasonable way.  In
principle one could choose a basis for the homogeneous harmonic polynomials
of degree $k$, integrate over the sphere to obtain the inner products of the
basis vectors in $W_k$, write down the evaluation map explicitly relative to
this basis, and solve simultaneous linear equations to obtain the
reproducing kernel.  However, that would be unpleasantly cumbersome.  The
beauty of the orthogonal polynomial characterization of ultraspherical
polynomials is that it is much more tractable, but we must still see why it
is true.

First, observe that $w_{k,x}$ and $w_{\ell,x}$ are orthogonal in
$L^2(S^{n-1})$ for $k \ne \ell$, since they are spherical harmonics of
different degrees. Thus,
\begin{equation} \label{eq:orthogsurface}
\int_{S^{n-1}} P^n_k(\langle x,y \rangle) P^n_\ell(\langle x,y \rangle) \, d\mu(y) = 0,
\end{equation}
where $\mu$ is surface measure.  We can now obtain the orthogonality of the
ultraspherical polynomials from the following multivariate calculus
exercise:

\begin{exercise}
Prove that under orthogonal projection from the surface of the sphere onto a
coordinate axis, the measure $\mu$ projects to a constant times the measure
$(1-t^2)^{(n-3)/2} \, dt$ on $[-1,1]$.  (See \cite[p.~2434]{Cohn2010} for a
simple solution.)
\end{exercise}

If we apply this orthogonal projection onto the axis between the antipodal
points $\pm x$, then \eqref{eq:orthogsurface} becomes
\[
\int_{-1}^1 P^n_k(t) P^n_\ell(t) (1-t^2)^{(n-3)/2} \, dt = 0,
\]
as desired.

As a side comment, we can now see that $W_k$ is an irreducible representation
of $O(n)$.  If it broke up further, then each summand would have its own
reproducing kernel, which would yield two different polynomials of degree $k$
that would be orthogonal to each other as well as to lower degree
polynomials. That's impossible, since the space of polynomials of degree at
most $k$ has dimension too low to contain so many orthogonal polynomials.

All that remains to prove is the Delsarte inequalities.  The key observation
is that $P^n_k(\langle x,y \rangle)$ can be written as the inner product of
two vectors in $W_k$ depending only on $x$ and $y$, namely the reproducing
kernels:

\begin{lemma} \label{lemma:Gram}
For all $x,y \in S^{n-1}$ and $k \ge 0$,
\[
P^n_k(\langle x,y \rangle) = \langle w_{k,x},w_{k,y} \rangle.
\]
\end{lemma}

\begin{proof}
Recall that the reproducing kernel property means $\langle w_{k,x},f \rangle
= f(x)$ for all $f \in W_k$.  In particular, taking $f = w_{k,y}$ yields
$\langle w_{k,x},w_{k,y} \rangle = w_{k,y}(x)$. Now $w_{k,y}(x) =
P^n_k(\langle x,y \rangle)$ implies that
\[
P^n_k(\langle x,y \rangle) = \langle w_{k,x},w_{k,y} \rangle,
\]
as desired.
\end{proof}

\begin{corollary}
For every finite subset $\mC \subset S^{n-1}$ and $k \ge 0$,
\[
\sum_{x,y \in \mC} P^n_k(\langle x,y \rangle) \ge 0.
\]
\end{corollary}

\begin{proof}
We have
\begin{align*}
\sum_{x,y \in \mC} P^n_k(\langle x,y \rangle)
&= \sum_{x,y \in \mC} \langle w_{k,x},w_{k,y} \rangle\\
&= \left|\sum_{x \in \mC} w_{k,x} \right|^2\\
&\ge 0,
\end{align*}
as desired.
\end{proof}

This argument is a perfect generalization of $\left|\sum_{x \in
\mC}x\right|^2 \ge 0$, except instead of summing the vectors $x$, we are
summing vectors $w_{k,x}$ in the Hilbert space $W_k$.  One interpretation is
that $x \mapsto w_{k,x}$ maps $S^{n-1}$ into a sphere in the
higher-dimensional space $W_k$, and we're combining the trivial inequality
\[
\left|\sum_{x \in \mC} w_{k,x} \right|^2 \ge 0
\]
with that nontrivial mapping.  When $n=2$, the space $W_k$ has dimension $2$
for $k \ge 1$, and so up to scaling we are mapping $S^1$ to itself.  This map
wraps $S^1$ around itself $k$ times, while the analogues for $n \ge 3$ are
more subtle.

It's natural to wonder whether ultraspherical polynomials span all the
functions $P$ satisfying
\[
\sum_{x,y \in \mC} P(\langle x,y \rangle) \ge 0
\]
for all $\mC$.  In fact, they do not.  Pfender has constructed further such
functions and used them to obtain improvements on linear programming bounds
\cite{Pfender2007}.  However, the numerical improvements so far have been
relatively modest.

Instead, Schoenberg proved that the ultraspherical polynomials span the
space of \emph{positive-definite kernels} \cite{Schoenberg1942}, i.e.,
functions $P$ such that for all $x_1,\dots,x_N \in S^{n-1}$, the $N \times
N$ matrix with entries $P(\langle x_i,x_j \rangle)$ is positive
semidefinite.  The reason why ultraspherical polynomials are
positive-definite is Lemma~\ref{lemma:Gram}: the matrix with entries
$\langle w_{k,x_i},w_{k,x_j} \rangle$ is a Gram matrix and is thus positive
semidefinite. Positive-definite kernels play an important role in
representation theory, which contributes to the importance of ultraspherical
polynomials.

As a final comment, everything we have done in this lecture has been
restricted to analyzing pairwise distance distributions.  It's natural to ask
what happens if one looks at triples of points instead of pairs, or even
larger subconfigurations.  The Delsarte inequalities can be generalized to
semidefinite constraints on these higher-order correlation functions, and
thus we can obtain \emph{semidefinite programming bounds}
\cite{Schrijver2005,BachocVallentin2008,deLaatVallentin2013}, which are a
powerful and important extension of linear programming bounds. For reasons
that have not yet been understood, these higher-order bounds seem less
fruitful for obtaining sharp bounds, but several sharp cases are known
\cite{BachocVallentin2009,CohnWoo2012} and others presumably remain to be
discovered.

\lecture{Packing bounds in Euclidean space}

\section{Introduction}

In this lecture we will study linear programming bounds for the sphere
packing problem in Euclidean space.  The basic principles are closely
analogous to those we saw for spherical codes in the fourth lecture.
However, the way the bounds behave in Euclidean space is far more
mysterious.  They almost certainly solve the sphere packing problem in
$\R^8$ and $\R^{24}$, by matching the densities of the $E_8$ and Leech
lattices, but nobody has been able to prove it.

We will focus on the sphere packing problem, rather than energy
minimization. Everything we will do works just as well in the latter case
(see \S9 of \cite{CohnKumar2007}), but sphere packing already illustrates
the essential features of these bounds.

To begin, let's review the statement and proof of linear programming bounds
for spherical codes, i.e., Theorem~\ref{thm:sphercode} from the last
lecture:

\begin{theorem} \label{thm:sphercoderepeat}
Suppose $h = \sum_k h_k P^n_k$ with $h_k \ge 0$ for $k \ge 0$ and $h_0>0$,
and suppose $h(t) \le 0$ for $t \in [-1,\cos \theta]$. Then every code $\mC$
in $S^{n-1}$ with minimal angle at least $\theta$ satisfies
\[
|\mC| \le h(1)/h_0.
\]
\end{theorem}

\begin{proof}
We have
\[
|\mC| h(1) \ge \sum_{x,y \in \mC} h(\langle x,y \rangle) = \sum_k h_k \sum_{x,y \in \mC} P^n_k(\langle x,y \rangle) \ge |\mC|^2 h_0. \qedhere
\]
\end{proof}

How could we generalize this argument?  First, we need functions on
Euclidean space that can play the same role as ultraspherical polynomials.
In particular, we need an analogue of the positivity property
\[
\sum_{x,y \in \mC} P^n_k(\langle x,y \rangle) \ge 0.
\]
As it turns out, the Euclidean functions are considerably more familiar,
namely exponentials $x \mapsto e^{2\pi i \langle t,x \rangle}$.  If we apply
them to two points via $(x,y) \mapsto e^{2\pi i \langle t,x-y \rangle}$, then
for every finite subset $\mC$ of $\R^n$,
\begin{equation} \label{eq:exppos}
\sum_{x,y \in \mC} e^{2\pi i \langle t,x-y \rangle} = \left|\sum_{x \in \mC} e^{2\pi i \langle t,x \rangle}\right|^2 \ge 0.
\end{equation}
As in the third lecture, these functions have representation-theoretic
origins, but we will not take up that subject here.

In the same way we previously made use of nonnegative linear combinations of
ultraspherical polynomials, we will now need to use nonnegative linear
combinations of exponentials. The natural setting for linear combinations of
exponentials is the Fourier transform. Define the \emph{Fourier transform}
$\widehat{f}$ of an integrable function $f \colon \R^n \to \R$ by
\[
\widehat{f}(t) = \int_{\R^n} f(x) e^{-2\pi i \langle t,x \rangle} \, dx.
\]
If $f$ is continuous and $\widehat{f}$ is integrable as well, then Fourier
inversion tells us that
\[
f(x) = \int_{\R^n} \widehat{f}(t) e^{2 \pi i \langle t,x \rangle} \, dt.
\]
In other words, the Fourier transform $\widehat{f}$ gives the coefficients
needed to express $f$ as a continuous linear combination of exponentials.
Thus, we will be particularly interested in functions $f$ for which
$\widehat{f} \ge 0$.

If $\widehat{f}(t) \ge 0$ for all $t$, then Fourier inversion implies that
\[
\sum_{x,y \in \mC} f(x-y) \ge 0
\]
whenever $\mC$ is a finite subset of $\R^n$, because
\[
\sum_{x,y \in \mC} f(x-y) = \int_{\R^n} \widehat{f}(t) \left|\sum_{x \in \mC} e^{2\pi i \langle t,x \rangle}\right|^2 dt
\]
by \eqref{eq:exppos}.  Thus, functions with nonnegative Fourier transforms
have exactly the property we need to generalize the Delsarte inequalities to
Euclidean space.

However, using these functions to prove sphere packing bounds requires some
finesse.  In the spherical case, we looked at the double sum
\[
\sum_{x,y \in \mC} h(\langle x,y \rangle)
\]
and bounded it on both sides to get
\[
|\mC| h(1) \ge \sum_{x,y \in \mC} h(\langle x,y \rangle) = \sum_k h_k \sum_{x,y \in \mC} P^n_k(\langle x,y \rangle) \ge |\mC|^2 h_0.
\]
In Euclidean space, the corresponding double sum would be
\[
\sum_{x,y \in \mC} f(x-y),
\]
where $\mC$ is a dense sphere packing, or rather the set of sphere centers
in such a packing.  Unfortunately, there's an obvious problem with this
approach: $\mC$ will be infinite and the double sum will diverge.  For
example, if $\mC$ is a lattice, then every term in the sum occurs infinitely
often, because there are infinitely many ways to write each lattice vector
as a difference of lattice vectors.

Can we somehow renormalize the double sum and use it to complete the proof?
The answer is yes if we're careful; see the proof of Theorem~3.3 in
\cite{CohnZhao2014}, which controls the sum over a packing by subtracting a
uniform background distribution of equal density. However, in this lecture
we'll take an arguably more fundamental approach using the Poisson summation
formula.

\section{Poisson summation}

Poisson summation is a remarkable duality between summing a function over a
lattice and summing its Fourier transform over the dual lattice.

We'll take a somewhat cavalier attitude towards analytic technicalities: we
will manipulate sums and integrals however we like, and include enough
hypotheses to justify these manipulations.  Specifically, we will deal with
what we'll call \emph{admissible functions} $f \colon \R^n \to \R$, those for
which $|f(x)| = O\big((1+|x|)^{-n-\varepsilon}\big)$ and $|\widehat{f}(t)| =
O\big((1+|t|)^{-n-\varepsilon}\big)$ for some $\varepsilon>0$.  This decay
rate is fast enough for sums over lattices to converge with room to spare. In
practice, we can simply read ``admissible'' as ``sufficiently rapidly
decreasing and smooth for everything to work.''

\begin{theorem}[Poisson summation]
If $f \colon \R^n \to \R$ is an admissible function and $\Lambda$ is a
lattice in $\R^n$, then
\[
\sum_{x \in \Lambda} f(x) = \frac{1}{\vol(\R^n/\Lambda)} \sum_{t \in \Lambda^*} \widehat{f}(t).
\]
\end{theorem}

Here $\vol(\R^n/\Lambda)$ is the volume of a fundamental cell of $\Lambda$,
i.e., the determinant of $\Lambda$, and
\[
\Lambda^* = \{t \in \R^n \ |\  \langle x,t \rangle \in \Z \textup{ for all $x \in \Lambda$}\}
\]
is the dual lattice (see Exercise~\ref{ex:duallattice} in the third lecture).

\begin{proof}
The key idea is to prove an even more general formula, by looking at the
Fourier expansion of the periodization of $f$ under $\Lambda$.  Let
\[
F(y) = \sum_{x \in \Lambda} f(x+y),
\]
so that $F$ is periodic modulo $\Lambda$.  We can expand $F$ as a Fourier
series
\[
F(y) = \sum_{t \in \Lambda^*} c_t e^{2\pi i \langle t,y \rangle}
\]
for some coefficients $c_t$, where $\Lambda^*$ occurs because it specifies
the exponentials that are periodic modulo $\Lambda$ (see
\S\ref{sec:fouriertorus} in the third lecture).

Let $D$ be a fundamental domain for $\Lambda$. By orthogonality,
\begin{align*}
c_t &= \frac{1}{\vol(D)}\int_D F(y) e^{-2\pi i \langle t,y \rangle} \, dy\\
&= \frac{1}{\vol(\R^n/\Lambda)} \int_D F(y) e^{-2\pi i \langle t,y \rangle} \, dy\\
&= \frac{1}{\vol(\R^n/\Lambda)} \sum_{x \in \Lambda} \int_D f(x+y) e^{-2\pi i \langle t,y \rangle} \, dy\\
&= \frac{1}{\vol(\R^n/\Lambda)} \sum_{x \in \Lambda} \int_D f(x+y) e^{-2\pi i \langle t,x+y \rangle} \, dy & \text{($x \in \Lambda$ and $t \in \Lambda^*$)}\\
&= \frac{1}{\vol(\R^n/\Lambda)} \int_{\R^n} f(y) e^{-2\pi i \langle t,y \rangle} \, dy  & \text{(translates of $D$ tile $\R^n$)}\\
&= \frac{1}{\vol(\R^n/\Lambda)} \widehat{f}(t).
\end{align*}
In other words, the Fourier coefficients $c_t$ of the periodization of $f$
are simply proportional to $\widehat{f}(t)$, with constant of proportionality
$1/\vol(\R^n/\Lambda)$.

Thus,
\[
\sum_{x \in \Lambda} f(x+y) =  \frac{1}{\vol(\R^n/\Lambda)}\sum_{t \in \Lambda^*} \widehat{f}(t) e^{2\pi i \langle t,y \rangle},
\]
and setting $y=0$ yields Poisson summation.
\end{proof}

The more general formula
\[
\sum_{x \in \Lambda} f(x+y) =  \frac{1}{\vol(\R^n/\Lambda)}\sum_{t \in \Lambda^*} \widehat{f}(t) e^{2\pi i \langle t,y \rangle}
\]
is important in its own right, not just as tool for proving Poisson
summation. At first it looks considerably more general than Poisson
summation, but it is simply Poisson summation applied to the function $x
\mapsto f(x+y)$ in place of $f$.

\section{Linear programming bounds}

We can now state and prove the linear programming bounds for Euclidean
sphere packings.

\begin{theorem}[Cohn and Elkies \cite{CohnElkies2003}] \label{thm:euclidean}
Let $f \colon \R^n \to \R$ be an admissible function with $f(x) \le 0$ for
$|x| \ge 2$, $\widehat{f}(t) \ge 0$ for all $t$, and $\widehat{f}(0)>0$. Then
the sphere packing density in $\R^n$ is at most
\[
\frac{\pi^{n/2}}{(n/2)!} \cdot \frac{f(0)}{\widehat{f}(0)}.
\]
\end{theorem}

As before, $(n/2)!$ means $\Gamma(n/2+1)$ when $n$ is odd.  The factor of
$\pi^{n/2}/(n/2)!$ is the volume of a unit ball.  It occurs because we are
looking at packing density, rather than just the number of balls per unit
volume in space.

We will prove Theorem~\ref{thm:euclidean} using Poisson summation
\cite{CohnElkies2003}. Several other proofs are known, but they are longer
\cite{Cohn2002} or more delicate \cite{CohnZhao2014}.  One advantage of the
proof in \cite{CohnZhao2014} is that it weakens the admissibility
hypothesis, so that we can use a more robust space of functions; however, it
obscures when a sharp bound can be obtained.

Before we turn to the proof, let's compare Theorem~\ref{thm:euclidean} with
Theorem~\ref{thm:sphercoderepeat}, its spherical analogue.  One difference
is that the Euclidean case involves a function of $n$ variables, as opposed
to one variable in the spherical case.  However, this difference is
illusory: we might as well radially symmetrize $f$ in the Euclidean case
(since both the hypotheses and the bound are radially symmetric), after
which it becomes a function of one variable.

Table~\ref{tab:dict} gives a dictionary with which these theorems can be
compared. They really are fully analogous, with the biggest discrepancy
being that we use inner products to measure distances in the spherical case
but Euclidean distance in the Euclidean case.

\begin{table}
\caption{A dictionary for comparing linear programming bounds on spheres and in Euclidean space.}
\label{tab:dict}
\begin{center}
\begin{tabular}{lll}
\toprule
space & $S^{n-1}$ & $\R^n$\\
function & $h$ & $f$\\
transform & $h_k$ & $\widehat{f}(t)$\\
balls don't overlap & $t \in [-1,\cos\theta]$ & $|x| \ge 2$\\
value at distance zero & $h(1)$ & $f(0)$\\
bound & $h(1)/h_0$ & $f(0)/\widehat{f}(0)$\\
\bottomrule
\end{tabular}
\end{center}
\end{table}

\begin{proof}
As a warm-up, let's prove the linear programming bounds for lattice
packings. Suppose $\Lambda$ is a lattice packing with unit balls (since we
can specify the packing radius without loss of generality). In other words,
the minimal vector length of the lattice $\Lambda$ is at least $2$.

By Poisson summation,
\[
\sum_{x \in \Lambda} f(x) = \frac{1}{\vol(\R^n/\Lambda)} \sum_{t \in \Lambda^*} \widehat{f}(t).
\]
We will apply the contrasting inequalities $f(x) \le 0$ (for $|x| \ge 2$) and
$\widehat{f}(t) \ge 0$ to this identity. We have
\[
f(0) \ge \sum_{x \in \Lambda} f(x)
\]
because $f(x) \le 0$ for $|x| \ge 2$, while
\[
\sum_{t \in \Lambda^*} \widehat{f}(t) \ge \widehat{f}(0)
\]
because $\widehat{f}(t) \ge 0$ for all $t$.  Thus,
\[
f(0) \ge \frac{\widehat{f}(0)}{\vol(\R^n/\Lambda)}.
\]
The number of balls per unit volume in the packing is $1/\vol(\R^n/\Lambda)$,
and its density is therefore $1/\vol(\R^n/\Lambda)$ times the volume of a
unit ball. Thus, the density is at most
\[
\frac{\pi^{n/2}}{(n/2)!} \cdot \frac{f(0)}{\widehat{f}(0)},
\]
as desired.

So far, we have done nothing but apply the given inequalities to both sides
of Poisson summation. Handling general packings will require a little more
work, but nothing too strenuous.

Without loss of generality, we can restrict our attention to periodic
packings, since they come arbitrarily close to the optimal packing density.
In other words, we can suppose our packing consists of $N$ translates of a
lattice $\Lambda$, namely
\[
\Lambda+y_1,\dots,\Lambda+y_N.
\]
Now the number of balls per unit volume in the packing is
$N/\vol(\R^n/\Lambda)$, and the condition that they should not overlap says
that $|x+y_j-y_k| \ge 2$ for $x \in \Lambda$ as long as $x \ne 0$ or $j \ne
k$.

A little manipulation based on the translated Poisson summation formula
\[
\sum_{x \in \Lambda} f(x+y) =  \frac{1}{\vol(\R^n/\Lambda)}\sum_{t \in \Lambda^*} \widehat{f}(t) e^{2\pi i \langle t,y \rangle}
\]
shows that
\[
\sum_{x \in \Lambda} \sum_{j,k=1}^N f(x+y_j-y_k) =
\frac{1}{\vol(\R^n/\Lambda)} \sum_{t \in \Lambda^*} \widehat{f}(t) \left|\sum_{j=1}^N e^{2\pi i \langle y_j, t \rangle}\right|^2.
\]
The inequalities on $f$ and $\widehat{f}$ show that the left side is at most
$N f(0)$ and the right side is at least $N^2
\widehat{f}(0)/\vol(\R^n/\Lambda)$.  It follows that
\[
\frac{\pi^{n/2}}{(n/2)!} \cdot \frac{N}{\vol(\R^n/\Lambda)} \le \frac{\pi^{n/2}}{(n/2)!} \cdot \frac{f(0)}{\widehat{f}(0)},
\]
which completes the proof.
\end{proof}

In this proof, the inequality
\[
\sum_{x \in \Lambda} \sum_{j,k=1}^N f(x+y_j-y_k)
\ge \frac{N^2 \widehat{f}(0)}{\vol(\R^n/\Lambda)}
\]
for functions satisfying $\widehat{f} \ge 0$ plays the role of the Delsarte
inequalities.  Note that the left side is essentially summing $f$ over the
distance distribution of the packing, but renormalized so that the distances
do not occur infinitely often.

In physics terms \cite[p.~72]{Torquato2002}, this inequality says that the
structure factor (the Fourier transform of $g_2-1$, where $g_2$ is the pair
correlation function) is nonnegative.  The structure factor plays an
important role in the theory of scattering, and its nonnegativity is a
fundamental constraint on the pair correlations that can occur in any
material.

The Poisson summation approach to linear programming bounds generalizes
naturally to the Selberg trace formula \cite{Selberg56,EGM98}. Specifically,
one can use a pretrace formula to prove density bounds in hyperbolic space
\cite{CohnZhao2014}.  However, several things that are known in the
Euclidean case remain mysterious in hyperbolic geometry.  In particular, the
bounds based on the pretrace formula have been proved only for periodic
packings, which are not known to come arbitrarily close to the optimal
density in hyperbolic space, and it is not known how to decrease the
hypotheses on the auxiliary function $f$ along the lines of the proof for
Euclidean space in \cite{CohnZhao2014}.

\section{Optimization and conjectures}

As in the spherical case, nobody knows how to choose the optimal auxiliary
function in the linear programming bounds.  It is not difficult to obtain a
trivial bound as follows:

\begin{exercise} \label{ex:trivbound}
Let $\chi_B$ be the characteristic function of the unit ball centered at the
origin in $\R^n$. Show that the convolution $f = \chi_B * \chi_B$ satisfies
the hypotheses of Theorem~\ref{thm:euclidean} and yields an upper bound of
$1$ for the sphere packing density.
\end{exercise}

Despite the triviality of the bound, this function is of some interest
\cite{Siegel1935}, but better constructions are needed if we are to prove
nontrivial bounds.  See, for example, \S6 of \cite{CohnElkies2003} for
constructions based on Bessel functions.

The behavior of the optimized linear programming bound in $\R^n$ as $n \to
\infty$ is unclear.  Cohn and Zhao \cite{CohnZhao2014} showed that it is at
least as good as the Kabatiansky-Levenshtein bound of $2^{-(0.5990\ldots
+o(1))n}$, while Torquato and Stillinger \cite{TorquatoStillinger2006}
showed that it can be no better than $2^{-(0.7786\ldots+o(1))n}$.  In
particular, it comes nowhere near the density of $2^{-(1+o(1))n}$ attained
by the best sphere packings currently known, although it might come much
closer than the Kabatiansky-Levenshtein bound does. Aside from these
constraints, the asymptotics are a mystery.

We will focus instead on how the bounds behave in low dimensions, by which I
mean ``not tending to infinity'' rather than low in the everyday sense.
Linear programming bounds are nearly the best bounds known in four or more
dimensions (there is a small improvement based on incorporating one more term
from Poisson summation \cite{dLOFV2014}).  As shown in
Figure~\ref{fig:LPgraph}, they are not so far from the truth in eight or
fewer dimensions, but they gradually drift away from the current record
densities in high dimensions.  Note that the jaggedness of the record
densities reflects their subtle dependence on the dimension.

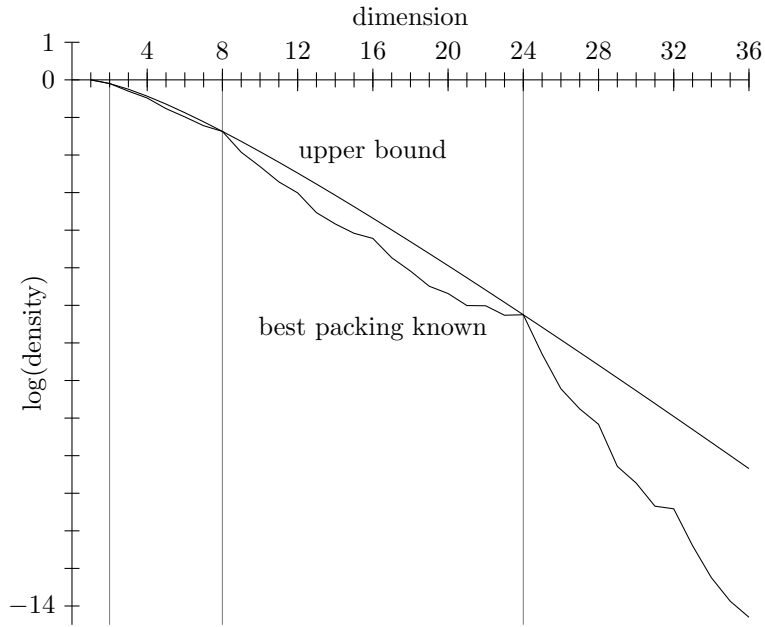
\begin{figure}
\begin{center}
\begin{tikzpicture}
\draw[black!50] (0.5,0)--(0.5,-7.25);
\draw[black!50] (2,0)--(2,-7.25);
\draw[black!50] (6,0)--(6,-7.25);
\foreach \x in {1,2,3,4,5,6,7,8,9} \draw (\x,-0.15)--(\x,0.15);
\foreach \x in {1,2,3,4,5,6,7,8,9,10,11,12,13,14,15,16,17,18,19,20,21,22,23,24,25,26,27,28,29,30,31,32,33,34,35,36} \draw ({0.25*\x},-0.1)--({0.25*\x},0.1);
\foreach \y in {-14,-13,-12,-11,-10,-9,-8,-7,-6,-5,-4,-3,-2,-1,0,1} \draw (-0.1,{0.5*\y})--(0.1,{0.5*\y});
\draw (4.5,0.6) node[above] {dimension};
\draw (-0.5,-3.5) node[rotate=90] {$\log(\text{density})$};
\draw (4,-0.95) node {upper bound};
\draw (4,-3.3) node {best packing known};
\draw (1,0.15) node[above] {$4$};
\draw (2,0.15) node[above] {$8$};
\draw (3,0.15) node[above] {$12$};
\draw (4,0.15) node[above] {$16$};
\draw (5,0.15) node[above] {$20$};
\draw (6,0.15) node[above] {$24$};
\draw (7,0.15) node[above] {$28$};
\draw (8,0.15) node[above] {$32$};
\draw (9,0.15) node[above] {$36$};
\draw (-0.1,0.5) node[left] {$1$};
\draw (-0.1,0) node[left] {$0$};
\draw (-0.1,-7) node[left] {$-14$};
\draw (0,-7.25)--(0,0.5);
\draw (0,0)--(9,0);
\draw (0.25,0)--(0.50000,-0.048862)--(0.75000,-0.12439)--(1.0000,-0.21716)--(1.2500,-0.32220)--(1.5000,-0.43653)
--(1.7500,-0.55820)--(2.0000,-0.68586)--(2.2500,-0.81852)--(2.5000,-0.95543)--(2.7500,-1.0960)--(3.0000,-1.2398)
--(3.2500,-1.3864)--(3.5000,-1.5356)--(3.7500,-1.6871)--(4.0000,-1.8406)--(4.2500,-1.9960)--(4.5000,-2.1531)
--(4.7500,-2.3118)--(5.0000,-2.4719)--(5.2500,-2.6334)--(5.5000,-2.7962)--(5.7500,-2.9602)--(6.0000,-3.1252)
--(6.2500,-3.2913)--(6.5000,-3.4584)--(6.7500,-3.6264)--(7.0000,-3.7952)--(7.2500,-3.9649)--(7.5000,-4.1354)
--(7.7500,-4.3066)--(8.0000,-4.4785)--(8.2500,-4.6510)--(8.5000,-4.8243)--(8.7500,-4.9981)--(9.0000,-5.1725);
\draw (0.25,0)--(0.50000,-0.048853)--(0.75000,-0.15022)--(1.0000,-0.24156)--(1.2500,-0.38257)--(1.5000,-0.49315)
--(1.7500,-0.60989)--(2.0000,-0.68586)--(2.2500,-0.96289)--(2.5000,-1.1532)--(2.7500,-1.3572)--(3.0000,-1.5033)
--(3.2500,-1.7666)--(3.5000,-1.9170)--(3.7500,-2.0415)--(4.0000,-2.1097)--(4.2500,-2.3659)--(4.5000,-2.5442)
--(4.7500,-2.7458)--(5.0000,-2.8428)--(5.2500,-3.0026)--(5.5000,-3.0056)--(5.7500,-3.1316)--(6.0000,-3.1252)
--(6.2500,-3.6488)--(6.5000,-4.1100)--(6.7500,-4.3777)--(7.0000,-4.5825)--(7.2500,-5.1425)--(7.5000,-5.3642)
--(7.7500,-5.6721)--(8.0000,-5.7070)--(8.2500,-6.1973)--(8.5000,-6.6231)--(8.7500,-6.9354)--(9.0000,-7.1499);
\end{tikzpicture}
\end{center}
\caption{The logarithm of sphere packing density as a function of dimension.  The upper curve is the
linear programming bound, while the lower curve is the best packing currently known. Vertical lines mark
conjectured equality above one dimension.}
\label{fig:LPgraph}
\end{figure}

The most remarkable feature of Figure~\ref{fig:LPgraph} is that the curves
appear to touch in eight and twenty-four dimensions.  If true, this would
settle the sphere packing problem in those dimensions, without the
difficulties that plague three dimensions.

\begin{conjecture}[Cohn and Elkies \cite{CohnElkies2003}]
The linear programming bounds for sphere packing density in $\R^n$ are sharp
when $n=2$, $8$, or $24$.
\end{conjecture}

Equality holds to at least fifty decimal places \cite{CohnMiller2016}, but no
proof is known.\footnote{Viazovska \cite{Viazovska2016} has recently proved
the conjecture for $n=8$.} It is furthermore conjectured that the linear
programming bounds for energy are sharp, which would lead to universal
optimality in Euclidean space \cite[\S9]{CohnKumar2007}.

Examining the proof of Theorem~\ref{thm:euclidean} shows that the auxiliary
function $f$ proves a sharp bound for a lattice $\Lambda$ iff $f(x) = 0$ for
all $x \in \Lambda\setminus\{0\}$ and $\widehat{f}(t)=0$ for all $t \in
\Lambda^* \setminus \{0\}$.  In other words, all we have to do is to ensure
that $f$ and $\widehat{f}$ have certain roots without developing any
unwanted sign changes.  That sounds like a manageable problem, but
unfortunately it seems difficult to control the roots of a function and its
Fourier transform simultaneously.

Linear programming bounds seem not to be sharp in $\R^n$ except when $n=1$,
$2$, $8$, or $24$.  We can't rule out the possibility of sharp bounds in
other dimensions, but nobody has been able to identify any plausible
candidates.  The $n=1$ case follows from Exercise~\ref{ex:trivbound}, but
sharpness is not known even for $n=2$, let alone $8$ or $24$.  That makes it
seem all the more mysterious: eight or twenty-four dimensions could be truly
deep, but two dimensions cannot transcend human understanding.

The strongest evidence for the sharpness of these bounds is the numerics,
but there are also analogies with related bounds that are known to be sharp
in these dimensions, such as those for the kissing problem
\cite{Levenshtein1979,OdlyzkoSloane1979}.  It is worth noting that the
kissing bounds are not just sharp, but sharp to an unnecessary degree.
Because the kissing number is an integer, any bound with error less than $1$
could be truncated to the exact answer, but that turns out not to be
necessary. In particular, the bound in Theorem~\ref{thm:sphercoderepeat} is
generally not an integer, but for the kissing problem it miraculously turns
out to be integral in $\R^8$ and $\R^{24}$: it is exactly $240$ in $\R^8$
and exactly $196560$ in $\R^{24}$. This unexpected exactness raises the
question of whether sharp bounds hold also for quantities such as packing
density, where integrality does not apply, and indeed they do seem to.

The computations behind Figure~\ref{fig:LPgraph} are based on numerical
optimization within a restricted class of auxiliary functions. Specifically,
they use functions of the form $f(x) = p(|x|^2) e^{-\pi |x|^2}$, where $p$
is a polynomial of one variable.  Such functions are relatively tractable,
while being dense among all reasonable radial functions.

To carry out explicit computations, it is convenient to write $p$ in terms
of an eigenbasis of the Fourier transform:

\begin{exercise}
Let
\[
\mathcal{P}_k = \left\{\text{functions $x \mapsto p(|x|^2) e^{-\pi |x|^2}$ on $\R^n$}
\ \big|\ \text{$p$ is a polynomial and $\deg(p) \le k$}\right\}.
\]
Prove that $\mathcal{P}_k$ is closed under the Fourier transform.  What are
the eigenvalues of the Fourier transform on $\mathcal{P}_k$?  Show that the
polynomials $p$ corresponding to an eigenbasis are orthogonal with respect to
a certain measure on $[0,\infty)$, and compute that measure.
\end{exercise}

The most general approach to optimizing the linear programming bounds over
$\mathcal{P}_k$ is to impose the sign conditions on $f$ and $\widehat{f}$
via sums of squares and then optimize using semidefinite programming
\cite{Parrilo2003}. This technique will produce the best possible polynomial
$p$ of any given degree.  (Another approach is to force roots for $f$ and
$\widehat{f}$ and then optimize the root locations \cite{CohnElkies2003}.)
However, we cannot obtain a sharp bound by using polynomials, because they
have only finitely many roots.  The best we can do is to approximate the
optimal bound.

By contrast, in the spherical case we can prove sharp bounds using
polynomials.  That is one reason why linear programming bounds are so much
more tractable for spheres than they are in Euclidean space.

Numerical computations have thus far shed little light on the
high-dimensional asymptotics of linear programming bounds. These bounds are
difficult to compute precisely in high dimensions, because such computations
seem to require using high-degree polynomials. Computing the linear
programming bound in $\R^{1000}$ to fifteen decimal places would be an
impressive benchmark, which might be possible but would not be easy.  Even
just a few decimal places would be interesting, as would an order of
magnitude estimate in $\R^{10000}$.

Even without being able to prove that the linear programming bounds are sharp
in $\R^8$ and $\R^{24}$, they can be combined with further arguments to prove
optimality among lattices:

\begin{theorem}[Cohn and Kumar \cite{CohnKumar2009}]
The Leech lattice is the unique densest lattice in $\R^{24}$.
\end{theorem}

See also the exposition in \cite{CohnKumar2004}.  Aside from $\R^{24}$, the
optimal lattices are known only in up to eight dimensions
\cite{Blichfeldt1935,Vetcinkin1982}.

Cohn and Miller observed that the hypothetical auxiliary functions proving
sharp bounds in $\R^8$ and $\R^{24}$ have additional structure, which has
not yet been explained.  The patterns are prettiest if we rescale the
function $f$ and its input so that $f(0)=\widehat{f}(0)=1$, in which case
the linear programming bounds amount to minimizing the radius $r$ such that
$f(x) \le 0$ for $|x| \ge r$ (see Theorem~3.2 in \cite{CohnElkies2003}).
Then the quadratic Taylor coefficients appear to be rational numbers:

\begin{conjecture}[Cohn and Miller \cite{CohnMiller2016}]
The quadratic Taylor coefficients of the optimal radial functions $f$ and
$\widehat{f}$, normalized as above with $f(0)=\widehat{f}(0)=1$, are
rational numbers when $n=8$ or $n=24$, as shown in
Table~\ref{table:coeffdata}.
\end{conjecture}

\begin{table}
\caption{Numerically computed Taylor series coefficients of the
hypothetical sphere packing functions in $\R^n$, normalized so $f(0) = \widehat{f}(0)=1$.}
\label{table:coeffdata}
\begin{center}
\begin{tabular}{ccccc}
\toprule
$n$ & function & order & coefficient & conjecture\\
\midrule $8$ & $f$ & $2$ & $-2.7000000000000000000000000000\dots$ & $-27/10$\\
$8$ & $\widehat{f}$ & $2$ & $-1.5000000000000000000000000000\dots$ & $-3/2$\\
$24$ & $f$ & $2$ & $-2.6276556776556776556776556776\dots$ & $-14347/5460$\\
$24$ & $\widehat{f}$ & $2$ & $-1.3141025641025641025641025641\dots$ & $-205/156$\\
$8$ & $f$ & $4$ & $\phantom{-}4.2167501240968298210998965628\dots$ & ?\\
$8$ & $\widehat{f}$ & $4$ & $-1.2397969070295980026220596589\dots$ & ?\\
$24$ & $f$ & $4$ & $\phantom{-}3.8619903167183007758184168473\dots$ & ?\\
$24$ & $\widehat{f}$ & $4$ & $-0.7376727789015322303799539712\dots$ & ?\\
\bottomrule
\end{tabular}
\end{center}
\end{table}

Because $f$ and $\widehat{f}$ are even functions, the odd-degree Taylor
coefficients vanish.  The fourth-degree coefficients shown in
Table~\ref{table:coeffdata} remain unidentified.  If all the coefficients
could be determined, it would open the door to solving the sphere packing
problem in $\R^8$ and $\R^{24}$.


\begin{thebibliography}{99}
\bibitem{Ajtai1998} M.~Ajtai, \emph{The shortest vector problem in $L_2$ is
    NP-hard for randomized reductions}, in \emph{Proceedings of the Thirtieth Annual
    ACM Symposium on Theory of Computing}, pp.\ 10--19, Association for Computing Machinery, New York,
    1998.

\bibitem{Andreev1996} N.~N.~Andreev, \emph{An extremal property of
    the
    icosahedron}, East J.\ Approx.\ \textbf{2} (1996), 459--462.

\bibitem{Andreev1997} N.~N.~Andreev, \emph{Location of points on a sphere
    with minimal energy}, Proc.\ Steklov Inst.\ Math.\ \textbf{219} (1997), 20--24.

\bibitem{ABR2001} S.~Axler, P.~Bourdon, and W.~Ramey, \emph{Harmonic
    function theory}, second edition, Graduate Texts in Mathematics \textbf{137},
    Springer-Verlag, New York, 2001.

\bibitem{BachocVallentin2008} C.~Bachoc and F.~Vallentin, \emph{New upper
    bounds for kissing numbers from semidefinite programming},
    J.\ Amer.\ Math.\ Soc.\ \textbf{21} (2008), 909--924.

\bibitem{BachocVallentin2009} C.~Bachoc and F.~Vallentin, \emph{Optimality
    and
    uniqueness of the $(4,10,1/6)$ spherical code}, J.\ Combin.\ Theory Ser.\ A \textbf{116}
    (2009), 195--204.

\bibitem{Ball1992} K.~Ball, \emph{A lower bound for the optimal
    density
    of lattice packings}, Internat.\ Math.\ Res.\ Notices \textbf{1992},
    217--221.

\bibitem{BBCGKS2009} B.~Ballinger, G.~Blekherman, H.~Cohn,
    N.~Giansiracusa,
    E.~Kelly, and A.~Sch\"urmann, {\it Experimental
    study of energy-minimizing point configurations on
    spheres}, Experiment.\ Math.\ {\bf 18} (2009), 257--283.

\bibitem{Blichfeldt1935} H.~F.~Blichfeldt, \emph{The minimum values of
    positive
    quadratic forms in six, seven and eight variables}, Math.\ Z.\ \textbf{39}
    (1935), 1--15.

\bibitem{BHS2014} A.~V.~Bondarenko, D.~P.~Hardin, and E.~B.~Saff,
    \emph{Mesh ratios for best-packing and limits of minimal energy
    configurations}, Acta Math.\ Hungar.\ \textbf{142} (2014), 118--131.

\bibitem{BowickGiomi2009} M.~Bowick and L.~Giomi, \emph{Two-dimensional
    matter: order, curvature and defects}, Adv.\ in Phys.\ \textbf{58} (2009),
    449--563.

\bibitem{Cohn2002} H.~Cohn, \emph{New upper bounds on
    sphere packings II}, Geom.\ Topol.\ \textbf{6} (2002), 329--353.

\bibitem{Cohn2010} H.~Cohn, \emph{Order and disorder in
    energy
    minimization}, Proceedings of
    the International Congress of Mathematicians, Hyderabad, August 19--27,
    2010, Volume IV, pages 2416--2443, Hindustan Book Agency, New Delhi,
    2010.

\bibitem{CCEK2007} H.~Cohn, J.~H.~Conway, N.~D.~Elkies, and
    A.~Kumar, \emph{The $D_4$ root system is not universally optimal},
    Experiment.\ Math.\ \textbf{16} (2007), 313--320.

\bibitem{CohnElkies2003} H.~Cohn and N.~D.~Elkies, \emph{New upper bounds on
    sphere packings I}, Ann.\ of Math.\ (2) \textbf{157} (2003), 689--714.

\bibitem{CohnKumar2004} H.~Cohn and A.~Kumar, \emph{The
    densest lattice in twenty-four dimensions}, Electron.\
    Res.\ Announc.\ Amer.\ Math.\ Soc.\ \textbf{10} (2004), 58--67.

\bibitem{CohnKumar2007} H.~Cohn and A.~Kumar, \emph{Universally optimal
    distribution of points on spheres}, J.\ Amer.\ Math.\ Soc.\ \textbf{20} (2007), 99--148.

\bibitem{CohnKumar2009} H.~Cohn and A.~Kumar, \emph{Optimality and
    uniqueness of the Leech lattice among lattices},
    Ann.\ of Math.\ (2) \textbf{170} (2009), 1003--1050.

\bibitem{CohnKumar2009b} H.~Cohn and A.~Kumar, \emph{Algorithmic design of
    self-assembling structures}, Proc.\ Natl.\ Acad.\
    Sci.\ USA \textbf{106} (2009), 9570--9575.

\bibitem{CohnKumarMinton} H.~Cohn, A.~Kumar, and G.~Minton, \emph{Optimal
    simplices and codes in projective spaces}, to appear in Geometry and
    Topology, arXiv:1308.3188.

\bibitem{CohnMiller2016} H.~Cohn and S.~D.~Miller, \emph{Some
    properties of optimal functions for sphere packing in dimensions $8$
    and $24$},
    preprint, 2016, arXiv:1603.04759.

\bibitem{CohnWoo2012} H.~Cohn and J.~Woo, \emph{Three-point bounds for
    energy
    minimization}, J.\ Amer.\ Math.\ Soc.\ \textbf{25}
    (2012), 929--958.

\bibitem{CohnZhao2014} H.~Cohn and Y.~Zhao, \emph{Sphere packing bounds via
    spherical codes}, Duke Math.\ J.\ \textbf{163} (2014), 1965--2002.

\bibitem{CohnZhao2014b} H.~Cohn and Y.~Zhao, \emph{Energy-minimizing
    error-correcting codes}, IEEE Trans.\ Inform.\ Theory \textbf{60} (2014), 7442--7450.

\bibitem{ConwaySloane1995} J.~H.~Conway and N.~J.~A.~Sloane, \emph{What
    are all the best sphere packings in low dimensions?}, Discrete Comput.\ Geom.\ \textbf{13} (1995), 383--403.

\bibitem{ConwaySloane1999} J.~H.~Conway and N.~J.~A.~Sloane, \emph{Sphere
    packings, lattices and groups}, third edition,
    Grundlehren der Mathematischen Wissenschaften \textbf{290},
    Springer, New York, 1999.

\bibitem{CoxeterFewRogers1959} H.~S.~M.~Coxeter, L.~Few, and C.~A.~Rogers,
    \emph{Covering space with equal spheres}, Mathematika \textbf{6} (1959), 147--157.

\bibitem{Delsarte1972} P.~Delsarte, \emph{Bounds for unrestricted codes, by
    linear programming}, Philips Res.\ Rep.\ \textbf{27}
    (1972), 272--289.

\bibitem{DGS1977} P.~Delsarte, J.~M.~Goethals, and J.~J.~Seidel,
    \emph{Spherical codes and designs}, Geom.\ Dedicata \textbf{6}
    (1977), 363--388.

\bibitem{DLT2002} P.~D.~Dragnev, D.~A.~Legg, and D.~W.~Townsend,
    \emph{Discrete logarithmic energy on the sphere},
    Pacific J.\ Math.\ \textbf{207} (2002), 345--358.

\bibitem{Dyson1962} F.~J.~Dyson, \emph{A Brownian-motion model for the
    eigenvalues of a random matrix}, J.\ Math.\ Phys.\ \textbf{3} (1962),
    1191--1198.

\bibitem{Ebeling2013} W.~Ebeling, \emph{Lattices and codes: a course
    partially based on lectures by Friedrich Hirzebruch}, third edition,
    Advanced Lectures in Mathematics, Springer Spektrum, Wiesbaden, 2013.

\bibitem{EGM98} J.~Elstrodt, F.~Grunewald, and J.~Mennicke,
    \emph{Groups acting on hyperbolic space: harmonic
    analysis and number theory}, Springer-Verlag, Berlin, 1998.

\bibitem{FejesToth1964} L.~Fejes T\'oth, \emph{Regular Figures}, Pergamon
    Press, Macmillan, New York, 1964.

\bibitem{FejesToth1972} L.~Fejes T\'oth, \emph{Lagerungen in der Ebene auf
    der Kugel und im Raum}, second edition, Springer, Berlin, 1972.

\bibitem{GGH1997} O.~Goldreich, S.~Goldwasser, and S.~Halevi,
    \emph{Public-key cryptosystems from lattice reduction problems}, in \emph{Advances
    in Cryptology -- CRYPTO '97}, Lecture Notes in Computer Science, volume 1294, pp.\ 112--131, Springer-Verlag,
    Berlin, 1997.

\bibitem{Groemer1963} H.\ Groemer, \emph{Existenzs\"atze f\"ur
    Lagerungen
    im
    Euklidischen Raum}, Math.\ Z.\ \textbf{81} (1963), 260--278.

\bibitem{Hales2000} T.~C.~Hales, \emph{Cannonballs and honeycombs}, Notices
    Amer.\ Math.\ Soc.\ \textbf{47} (2000), 440--449.

\bibitem{Hales2005} T.~C.~Hales, \emph{A proof of the Kepler
    conjecture},
    Ann.\ of Math.\ (2) \textbf{162} (2005), 1065--1185.

\bibitem{Flyspeck} T.~Hales, M.~Adams, G.~Bauer, D.~T.~Dang, J.~Harrison,
    T.~L.~Hoang, C.~Kaliszyk, V.~Magron, S.~McLaughlin, T.~T.~Nguyen,
    T.~Q.~Nguyen, T.~Nipkow, S.~Obua, J.~Pleso, J.~Rute,
    A.~Solovyev, A.~H.~T.~Ta, T.~N.~Tran, D.~T.~Trieu, J.~Urban, K.~K.~Vu, and
    R.~Zumkeller, \emph{A formal proof of the Kepler
    conjecture}, preprint, 2015, arXiv:1501.02155.

\bibitem{KabatianskyLevenshtein1978} G.~A.~Kabatiansky and V.~I.~Levenshtein,
    \emph{Bounds for packings on a sphere and in space}, Probl.\ Inf.\ Transm.\
    \textbf{14} (1978), 1--17.

\bibitem{Kallus2013} Y.~Kallus, \emph{Statistical mechanics of the
    lattice sphere packing problem}, Phys.\ Rev.\ E \textbf{87} (2013), 063307, 5 pp.

\bibitem{KallusElserGravel2010} Y.~Kallus, V.~Elser, and S.~Gravel,
    \emph{Method for dense packing discovery}, Phys.\ Rev.\ E \textbf{82}
    (2010), 056707, 14 pp.

\bibitem{KolushovYudin1994} A.~V.~Kolushov and V.~A.~Yudin, \emph{On the
    Korkin-Zolotarev construction}, Discrete Math.\ Appl.\ {\bf 4}
    (1994), 143--146.

\bibitem{KolushovYudin1997} A.~V.~Kolushov and V.~A.~Yudin, \emph{Extremal
    dispositions of points on the sphere}, Anal.\ Math.\ {\bf 23}
    (1997), 25--34.

\bibitem{dLOFV2014} D.~de Laat, F.~M.~de Oliveira Filho, and F.~Vallentin,
    \emph{Upper bounds for packings of spheres of several radii},
    Forum Math.\ Sigma \textbf{2} (2014), e23, 42 pp.

\bibitem{deLaatVallentin2013} D.~de Laat and F.~Vallentin, \emph{A
    semidefinite
    programming hierarchy for packing problems in discrete geometry}, Math.\
    Program.\ \textbf{151} (2015), Ser.\ B, 529--553.

\bibitem{Leech1957} J.~Leech, \emph{Equilibrium of sets of particles
    on a sphere}, Math.\ Gaz.\ {\bf 41} (1957), 81--90.

\bibitem{LLL1982} A.~K.~Lenstra, H.~W.~Lenstra, Jr., and L.~Lov\'asz,
    \emph{Factoring polynomials with rational coefficients}, Math.\ Ann.\ \textbf{261} (1982),
    515--534.

\bibitem{Levenshtein1979} V.~I.~Levenshtein, \emph{On bounds for packings
    in $n$-dimensional Euclidean space}, Soviet Math.\ Dokl.\
    \textbf{20} (1979), 417--421.

\bibitem{Levenshtein1992} V.~I.~Levenshtein, \emph{Designs as maximum codes
    in polynomial metric spaces}, Acta Appl.\ Math.\ \textbf{29}
    (1992), 1--82.

\bibitem{Loewen2000} H.~L\"owen, \emph{Fun with hard spheres}, in K.~R.~Mecke
    and
    D.~Stoyan, eds.,
    \emph{Statistical physics and spatial statistics: the art of analyzing and modeling
spatial structures and pattern formation},
    Lecture Notes in Physics \textbf{554}, Springer, New York, 2000,
    pp.\ 295--331.

\bibitem{MarcotteTorquato2013} \'E.~Marcotte and S.~Torquato,
    \emph{Efficient linear programming algorithm to generate the densest
    lattice sphere packings}, Phys.\ Rev.\ E \textbf{87} (2013), 063303, 9 pp.

\bibitem{Minton2011} G.~Minton, unpublished notes, 2011.

\bibitem{Musin2008}  O.~Musin, \emph{The kissing number in four
    dimensions}, Ann.\ of Math.\ (2) \textbf{168} (2008), 1--32.

\bibitem{Nachbin1965} L.~Nachbin, \emph{The Haar integral}, D.\ Van Nostrand
    Company, Inc., Princeton, NJ, 1965.

\bibitem{Nguyen1999} P.~Nguyen, \emph{Cryptanalysis of the
    Goldreich-Goldwasser-Halevi cryptosystem from Crypto '97}, in
    \emph{Advances in Cryptology -- CRYPTO '99}, Lecture Notes in Computer
    Science, volume 1666, pp.~288--304, Springer-Verlag, Berlin, 1999.

\bibitem{NguyenVallee2010} P.~Q.~Nguyen and B.~Vall\'ee, eds., \emph{The LLL
    algorithm: survey and applications}, Springer-Verlag, Berlin, 2010.

\bibitem{OdlyzkoSloane1979} A.~M.~Odlyzko and N.~J.~A.~Sloane, \emph{New
    bounds
    on the number of unit spheres that can touch a unit sphere in
    $n$~dimensions}, J.\ Combin.\ Theory Ser.\ A \textbf{26}
    (1979), 210--214.

\bibitem{Parrilo2003} P.~A.~Parrilo, \emph{Semidefinite programming
    relaxations for semialgebraic problems}, Math.\ Program.\ \textbf{96} (2003), Ser.\ B,
    293--320.

\bibitem{Peikert2015} C.~Peikert, \emph{A decade of lattice cryptography},
    Cryptology ePrint Archive, Report 2015/939, \url{https://eprint.iacr.org/2015/939}.

\bibitem{Pfender2007} F.~Pfender, \emph{Improved Delsarte bounds for
    spherical codes in small dimensions}, J.\ Combin.\ Theory Ser.\ A
    \textbf{114} (2007), 1133--1147.

\bibitem{Rogers1958} C.~A.~Rogers, \emph{The packing of equal spheres},
    Proc.\ London Math.\ Soc.\ (3) \textbf{8} (1958), 609--620.

\bibitem{Schoenberg1942} I.~J.~Schoenberg, \emph{Positive definite
    functions
    on spheres}, Duke Math.\ J.\ \textbf{9} (1942), 96--108.

\bibitem{Schrijver2005} A.~Schrijver, \emph{New code upper bounds from
    the
    Terwilliger algebra and semidefinite programming}, IEEE Trans.\ Inform.\ Theory \textbf{51} (2005), 2859--2866.

\bibitem{SvdW1953} K.~Sch\"utte and B.~L.~van der Waerden, \emph{Das Problem
    der
    dreizehn Kugeln}, Math.\ Ann.\ \textbf{125} (1953), 325--334.

\bibitem{Schwartz2013} R.~E.~Schwartz, \emph{The
    five-electron case of Thomson's problem}, Exp.\ Math.\ \textbf{22}
    (2013), 157--186.

\bibitem{Schwartz2015} R.~E.~Schwartz, \emph{The triangular bi-pyramid
    minimizes a range of power law potentials}, preprint, 2015, arXiv:1512.04628.

\bibitem{Selberg56} A.~Selberg, \emph{Harmonic analysis and
    discontinuous
    groups in weakly symmetric Riemannian spaces with applications to
    Dirichlet series}, J.\ Indian Math.\ Soc.\ (N.S.) \textbf{20} (1956), 47--87.

\bibitem{Shannon1948} C.~E.~Shannon, \emph{A mathematical theory of
    communication},
    Bell System Tech.\ J.\ \textbf{27} (1948), 379--423 and 623--656.

\bibitem{Siegel1935} C.~L.~Siegel, \emph{\"Uber Gitterpunkte in convexen
    K\"orpern und ein damit zusammenh\"angendes Extremalproblem}, Acta
    Math.\ \textbf{65} (1935), 307--323.

\bibitem{Siegel1945} C.~L.~Siegel, \emph{A mean value theorem in
    geometry
    of
    numbers}, Ann.\ of Math.\ (2) \textbf{46} (1945), 340--347.

\bibitem{DSV2007} M.~Dutour Sikiri\'c, A.~Sch\"urmann, and F.~Vallentin,
    \emph{Classification of eight-dimensional perfect forms},
    Electron.\ Res.\ Announc.\ Amer.\ Math.\ Soc.\ \textbf{13} (2007),
    21--32.

\bibitem{Thue1892} A.~Thue, \emph{Om nogle geometrisk-taltheoretiske
    Theoremer},
    Forhandlingerne ved de Skandinaviske Naturforskeres \textbf{14} (1892),
    352--353.

\bibitem{Torquato2002} S.~Torquato, \emph{Random heterogeneous materials:
    microstructure and macroscopic properties}, Interdisciplinary Applied
    Mathematics \textbf{16}, Springer-Verlag, New York, 2002.

\bibitem{TorquatoStillinger2006} S.~Torquato and F.~Stillinger, \emph{New
    conjectural
    lower
    bounds on the optimal density of sphere packings}, Experiment.\ Math.\
    \textbf{15} (2006), 307--331.

\bibitem{Vance2011} S.~Vance, \emph{Improved sphere packing lower
    bounds from Hurwitz lattices}, Adv.\ Math.\  \textbf{227} (2011),
    2144--2156.

\bibitem{Venkatesh2013} A.~Venkatesh, \emph{A note on sphere
    packings in high dimension}, Int.\ Math.\ Res.\ Not.\
    \textbf{2013} (2013), 1628--1642.

\bibitem{Vetcinkin1982} N.~M.~Vet\v{c}inkin, \emph{Uniqueness of classes of
    positive quadratic forms on which values of the Hermite constant
    are attained for $6 \le n \le 8$}, Proc.\ Steklov Inst.\
    Math.\ {\bf 152} (1982), 37--95.

\bibitem{Viazovska2016} M.~S.~Viazovska, \textit{The sphere packing problem
    in dimensions $8$}, preprint, 2016, arXiv:1603.04246.

\bibitem{Yudin1993} V.~A.~Yudin, \emph{Minimum potential energy of a point
    system of charges}, Discrete Math.\ Appl.\ {\bf 3} (1993), 75--81.
\end{thebibliography}
\end{document}